\documentclass[final,10pt]{amsart}
\usepackage{amsmath,amsfonts,amssymb,amscd,verbatim}
\usepackage{mathtools}
\usepackage{color}
\usepackage{showkeys}
\usepackage{stmaryrd}

\usepackage{enumerate}
\usepackage{bm}
\usepackage[arrow,matrix,cmtip,curve]{xy}
\usepackage{url}

\usepackage{multicol}

\numberwithin{table}{section}

\numberwithin{equation}{section}
\theoremstyle{plain}

\newtheorem{theorem}[equation]{Theorem}

\newtheorem{corollary}[equation]{Corollary}
\newtheorem{lemma}[equation]{Lemma}
\newtheorem{proposition}[equation]{Proposition}

\theoremstyle{definition}
\newtheorem{definition}[equation]{Definition}

\newtheorem{example}[equation]{Example}
\newtheorem{hypothesis}[equation]{Hypothesis}

\newtheorem{remark}[equation]{Remark}

\newcommand{\hdet}{\operatorname{hdet}}

\newcommand{\inv}{^{-1}}
\newcommand{\iso}{\cong}
\newcommand{\kk}{\Bbbk}

\newcommand{\tensor}{\otimes}

\newcommand{\mb}{\mathbb}

\newcommand{\op}{\operatorname{op}}

\newcommand{\NN}{\mathbb N}
\newcommand{\RR}{\mathbb R}
\newcommand{\ZZ}{\mathbb Z}

\DeclareMathOperator\Aut{Aut}
\DeclareMathOperator\adj{adj}
\DeclareMathOperator\Ext{Ext}

\DeclareMathOperator\GKdim{GKdim}

\newenvironment{psmatrix}
  {\left(\begin{smallmatrix}}
  {\end{smallmatrix}\right)}

\title[Quivers supporting  twisted Calabi-Yau algebras]{Quivers supporting twisted Calabi-Yau algebras}

\author[Gaddis]{Jason Gaddis}
\address{Miami University, Department of Mathematics, 301 S. Patterson Ave., Oxford, Ohio 45056}
\email{gaddisj@miamioh.edu}

\author[Rogalski]{Daniel Rogalski}
\address{University of California, San Diego, Department of Mathematics, 9500 Gilman Dr. \# 0112, La Jolla, CA 92093-0112}
\email{drogalsk@math.ucsd.edu}

\thanks{Rogalski was partially supported by the NSF grant
DMS-1201572 and the NSA grant H98230-15-1-0317.}
\date{\today}

\begin{document}

\begin{abstract}
We consider graded twisted Calabi-Yau algebras of dimension 3 which are derivation-quotient algebras
of the form $A = \kk Q/I$, where $Q$ is a quiver and $I$ is an ideal of relations coming from taking partial derivatives of a twisted superpotential on $Q$.   We define the type $(M, P, d)$ of such an algebra $A$, where $M$ is the incidence matrix of the quiver, $P$ is the permutation matrix giving the action of the Nakayama automorphism of $A$ on the vertices of the quiver, and $d$ is the degree of the superpotential.  We study the question of what possible types
can occur under the additional assumption that $A$ has polynomial growth.  In particular, we are able to
give a nearly complete answer to this question when $Q$ has at most 3 vertices.
\end{abstract}

\maketitle

\section{Introduction}

Fix  an algebraically closed field $\kk$ of characteristic zero.   All algebras in this paper are assumed to be $\kk$-algebras.
Let $A^e = A \otimes_\kk A^{\op}$ be the \emph{enveloping algebra} of an algebra $A$.
Then left $A^e$-modules are identified with $(A, A)$-bimodules.
\begin{definition}
\label{def.CY}
An algebra $A$ is said to be \emph{twisted Calabi-Yau} (CY) of dimension $d$ if
\begin{enumerate}
\item the left $A^e$-module $A$ has a finite projective $A^e$-module resolution by finitely generated
projective modules; and
\item there exists an invertible $(A, A)$-bimodule $U$ such that
\[\Ext_{A^e}^i (A, A^e) \cong \begin{cases} 0 & i \neq d \\ U & i = d \end{cases}\]
as right $A^e$-modules.
\end{enumerate}
The algebra $A$ is \emph{Calabi-Yau} if the definition above holds with $U = A$.
\end{definition}

Calabi-Yau algebras were originally defined by Ginzburg \cite{Gi} and have been the focus of much research since.  Of particular interest are (twisted) Calabi-Yau algebras defined by (twisted) superpotentials on quivers, and that is the setting on which we focus in this paper.  Let $Q = (Q_0, Q_1)$ be a quiver with finite vertex set $Q_0$ and finite arrow set $Q_1$, with multiple arrows and loops allowed
(see \cite{ASS} for background on quivers).
Let $\kk Q$ be its path algebra.  Note that our convention is to compose paths from left to right.  Given a path $p = a_1 a_2 \dots a_n$ of $Q$ and $a \in Q_1$, the derivation operator $\partial_a$ is defined on paths as $\partial_a(p) = a\inv p = a_2 \dots a_n$ if
$a_1 = a$ and $0$ otherwise, and extended linearly.  Given an automorphism $\sigma$ of $\kk Q$, a \emph{$\sigma$-twisted superpotential}
$\omega$ is a homogeneous element of $\kk Q$ of degree $d$ that is invariant under the linear map which acts
on paths by $a_1a_2\cdots a_d \mapsto \sigma(a_d)a_1\cdots a_{d-1}$.
We simply call $\omega$ a superpotential when $\sigma$ is the identity.
(Note that in \cite{BSW} the definition requires $\omega$ to be $(-1)^d$-invariant.  We are following the sign convention used in \cite{MS} instead, which seems more appropriate for non-Koszul algebras).
Given a $\sigma$-twisted superpotential $\omega$ in $\kk Q$, the corresponding \emph{derivation-quotient algebra} is
$\kk Q/ (\partial_{a}(\omega) : a \in Q_1)$, that is, the path algebra with relations given by the partial derivatives of the
superpotential at all arrows in $Q$.

In nice  cases, such derivation-quotient algebras are twisted Calabi-Yau algebras of dimension $3$, and when this happens, we call
$\omega$ \emph{good}.  Deciding which twisted superpotentials are good is a sensitive problem.
For example, Bocklandt proved that for a given quiver and length $d$, if $Q$ has any good (non-twisted) superpotential of degree $d$, then all superpotentials outside of a set of measure $0$ in the space of degree $d$ superpotentials are good \cite[Corollary 4.4]{Bo}.

In this paper we always consider a path algebra $\kk Q$ as an $\mb{N}$-graded algebra $\kk Q = \bigoplus_{n \geq 0} \kk Q_n$ where $\kk Q_n$ is the $\kk$-span of paths of length $n$.  The focus on algebras defined by superpotentials is justified by the following result.
\begin{theorem}[{\cite[Theorem 6.8, Remark 6.9]{BSW}}]
\label{thm.bsw}
Let $A = \kk Q/I$ for a connected quiver $Q$ and a homogeneous ideal $I$.  If $A$ is twisted Calabi-Yau of dimension $3$, then
$I$ is generated by relations all of the same degree $r$ and in fact $A$ is $r$-Koszul.  Moreover, $A$ is a derivation-quotient algebra for some twisted superpotential of degree $r+1$.
\end{theorem}

When the quiver $Q$ has only one vertex, a twisted Calabi-Yau derivation-quotient algebra of dimension $3$  is the same as an Artin-Schelter regular algebra, as defined in Section~\ref{sec.motive} below \cite[Lemma 1.2]{RRZ}.
The classification of such algebras is known with the additional assumption that the algebras
have finite GK dimension \cite{ATV1,AS}.   Following this program, one would ultimately
like to have some kind of classification of twisted Calabi-Yau derivation-quotient algebras of dimension $3$ with finite GK dimension.
Since some quivers do not have any good twisted superpotentials, or do not lead to an algebra with finite GK dimension, the first question in such a project is which quivers can occur.  In this paper, we solve this problem for quivers with two or three vertices (though there remain a few minor gaps in our understanding of the three vertex case).

In fact, we track more information than just the quiver.  Suppose that $A = \kk Q/I$ is a twisted Calabi-Yau derivation-quotient algebra for some twisted superpotential $\omega$ of degree $d$.  Fix some labeling of the vertices in $Q_0$ by $1, 2, \dots, m$ and let $M = M_Q$ be the incidence matrix of $Q$, where $M_{ij}$ is equal to the number of arrows in $Q$ from $i$ to $j$.   In this setting, the bimodule $U$ in the definition
of twisted Calabi-Yau must have the form ${}^1 A^{\mu}$ for some automorphism $\mu$ of $A$, the \emph{Nakayama automorphism} (see
\cite[Proposition 5.2]{RR1}).  In fact, lifting $\mu$ to an automorphism of $\kk Q$, $\omega$ must be $\mu^{-1}$-twisted (see \cite[Theorem 4.4]{MS} for the case of one vertex, noting that for Mori and Smith $\mu^{-1}$ is the Nakayama automorphism).
The Nakayama automorphism $\mu$ of $A$ permutes the idempotent trivial paths $e_i$, and induces a permutation of $\{1, \dots, m \}$ we also call $\mu$, where $\mu(e_i) = e_{\mu(i)}$.   We associate a permutation matrix $P$ such that $P_{ij} = \delta_{\mu(i),j}$.  We call $(M, P, d)$ the \emph{type} of $A$.  In this paper our main concern is which types can occur when $Q$ has few vertices.  We do not attempt to classify the possible twisting automorphisms $\mu$, only their actions
on the vertices.  Note that if two types differ only by a relabeling of the vertices, we consider them the same.

In the predecessor paper \cite{RR1}, the authors studied how conditions on the quiver relate to the growth of a
derivation-quotient algebra.  The primary technique used is the matrix-valued Hilbert series associated to the adjacency matrix of a quiver.
For any graded algebra of the form $A = \kk Q/I$, with fixed labeling of the vertices of $Q$, we write the trivial paths in $Q$
as $e_1, \dots, e_m$, so that $A_0 = \kk e_1 + \dots + \kk e_m$, where the $e_i$ are pairwise orthogonal primitive idempotents.
We then define the matrix-valued Hilbert series of $A$ as the formal power series $h_A(t) = \sum_{n=0}^{\infty} H_n t^n \in M_m(\kk)\llbracket t \rrbracket$, where $(H_n)_{ij} = \dim_k e_i A_n e_j$.  There is a convenient formula for this Hilbert series in terms of the type.
\begin{theorem}[{\cite[Propositions 2.8, 8.2, 8.10]{RR1}}]
\label{thm.hilb}
Let $Q$ be a connected quiver, and let
$A = \kk Q /I$ be a graded twisted Calabi-Yau algebra of dimension $3$ of type $(M, P, d)$, with $\GKdim(A) = 3$.
\begin{enumerate}
\item The matrices $M$ and $P$ commute, and we have $h_A(t) = (p(t))^{-1}$, where
\[
	p(t) = I- M t+ PM^T t^{d-1}- P t^d \in M_m(\kk)[t].
\]
Every zero of $\det p(t) \in \kk[t]$ is a root of unity and $\det p(t)$ vanishes at $t = 1$ to order at least $3$.
\item If $M$ is a normal matrix, then either either $d = 3$ or $d = 4$,
and the spectral radius $\rho(M)$ of $M$ satisfies $\rho(M) = 6-d$.
\end{enumerate}
\end{theorem}
\noindent
We conjecture that any graded twisted Calabi-Yau algebra $\kk Q/I$ of dimension $n$ with finite GK dimension has $\GKdim(A) = n$, but this is not known in general. In particular, this holds for all known examples of graded twisted Calabi-Yau algebras, including Artin-Schelter regular algebras and those arising from the various constructions presented in Section \ref{sec.motive}. So, we will classify certain twisted Calabi-Yau algebras of dimension $3$ that also have $\GKdim(A) = 3$, but really this should be the classification of all such algebras with finite GK dimension.

For the remainder of the paper we focus on algebras $A$ satisfying the following assumptions.
\begin{hypothesis}
\label{hyp.scy}
Let $A \cong \kk Q/(\partial_a \omega : a \in Q_1)$ be a derivation-quotient algebra on the connected quiver $Q$ with $\sigma$-twisted superpotential $\omega$ of degree $d$, for some automorphism $\sigma$ of $\kk Q$.  Assume that $A$ is twisted Calabi-Yau of dimension $3$ with
$\GKdim(A) = 3$, that the number $m = |Q_0|$ of vertices is two or three, and that $d = 3$ or $d = 4$.
\end{hypothesis}
\noindent Note that when the type $(M, P, d)$ of $A$ has a normal matrix $M$, then by Theorem~\ref{thm.hilb}
the hypothesis $d = 3$ or $d = 4$ is a consequence of the other assumptions in Hypothesis~\ref{hyp.scy}. For Artin-Schelter regular algebras, the hypothesis $d = 3$ or $d = 4$ is extraneous and for this reason we suspect the same is true in general.

We now describe our strategy in this paper for finding the possible types of algebras satisfying Hypothesis~\ref{hyp.scy}.   
First, we classify those types $(M, P, d)$ that satisfy the necessary conditions coming from Theorem~\ref{thm.hilb} in order to be the type of such a twisted Calabi-Yau algebra.
This is a difficult problem only in the three vertex case where $P = I$, where we use the software program Maple to help whittle down the possibilities.
Second, for the list of such types, we either find an explicit example of a twisted Calabi-Yau algebra with that type, or rule it out for some other reason.
Most of the types that occur can be realized by a skew group algebra of an action of a group of order $2$ or $3$ on a Artin-Schelter regular algebra of
dimension $3$, where the corresponding quiver is the McKay quiver.  However, there is also an infinite family of types on three vertices which we need to realize using the process of \emph{quiver mutation}.   In Section~\ref{sec.motive} we review these two methods for obtaining twisted Calabi-Yau algebras from known ones, as well as Ore extensions, which we use in a few cases to realize types.
Section~\ref{sec.key} proves a few other key results we need in the analysis, especially a limitation on the number of loops at
a vertex, which we describe below.

In  Section~\ref{sec.2vert} we study the two-vertex case, which
can be handled by elementary means.
\begin{theorem}[Lemma~\ref{lem.2vert1}, Lemma~\ref{lem.2vert2}]
\label{thm.2vert}
The possible types $(M, P, d)$ of twisted Calabi-Yau algebras $A$ satisfying Hypothesis~\ref{hyp.scy} with $|Q_0| = 2$
are as follows:
\begin{gather*}
M \in \left\lbrace \begin{psmatrix}1 & 2 \\ 2 & 1\end{psmatrix},
\begin{psmatrix}2 & 1 \\ 1 & 2\end{psmatrix},
\begin{psmatrix}0 & 3 \\ 3 & 0\end{psmatrix} \right\rbrace,
	P = \begin{psmatrix} 0 & 1 \\ 1 & 0 \end{psmatrix}, d = 3; \qquad
M = \begin{psmatrix}1 & 1 \\ 1 & 1\end{psmatrix}, P = \begin{psmatrix} 0 & 1 \\ 1 & 0 \end{psmatrix}, d=4; \\
M \in \left\lbrace \begin{psmatrix}1 & 2 \\ 2 & 1\end{psmatrix},
	\begin{psmatrix}2 & 1 \\ 1 & 2\end{psmatrix} \right\rbrace, P= I, d = 3; \qquad
M \in \left\lbrace \begin{psmatrix}1 & 1 \\ 1 & 1\end{psmatrix},
    \begin{psmatrix}0 & 2 \\ 2 & 0\end{psmatrix} \right\rbrace, P = I, d =4.
\end{gather*}
Every type can be realized by a skew group algebra of the form $B \# \mb{Z}_2$ where
$B$ is Artin-Schelter regular of dimension $3$.
\end{theorem}

In Section~\ref{sec.3vert1} we move on to the three vertex case where the Nakayama automorphism does not act trivially on the vertices.
Since the permutation matrix $P$ is nontrivial in this case, and $M$ and $P$ commute by Theorem~\ref{thm.hilb}, this restricts the form of $M$
enough to again make the possible types satisfying all of the conditions in Theorem~\ref{thm.hilb} fairly limited and classifiable more or less by hand.
\begin{theorem}[Proposition~\ref{prop.3vert1}, Proposition~\ref{prop.3vert2}]
\label{thm.3vert-twist}
If $A$ is twisted Calabi-Yau satisfying Hypothesis~\ref{hyp.scy} with $|Q_0| = 3$ and type $(M, P, d)$ with $P \neq I$, then
the type is one of the following:
\begin{gather*}
M \in \left\lbrace
\begin{psmatrix}
	2 & 1 & 0 \\
	0&  2 & 1 \\
	1 & 0 & 2
\end{psmatrix},
\begin{psmatrix}
	1 & 0 & 2 \\
	2&  1 & 0 \\
	0 & 2 & 1
\end{psmatrix},
\begin{psmatrix}
	0 & 2 & 1 \\
	1&  0 & 2 \\
	2 & 1 & 0
\end{psmatrix},
\begin{psmatrix}
	1 & 1 & 1 \\
	1 & 1 & 1 \\
	1 & 1 & 1
\end{psmatrix} \right\rbrace,
P = \begin{psmatrix}
	0 & 1 & 0 \\
	0 & 0 & 1 \\
	1 & 0 & 0
\end{psmatrix}, d = 3;  \\
M \in \left\lbrace
\begin{psmatrix}
	0 & 1 & 1 \\
	1 & 0 & 1 \\
	1 & 1 & 0
\end{psmatrix},
\begin{psmatrix}
	1 & 1 & 0 \\
	0 & 1 & 1 \\
	1 & 0 & 1
\end{psmatrix},
\begin{psmatrix}
	1 & 0 & 1 \\
	1 & 1 & 0 \\
	0 & 1 & 1
\end{psmatrix},
\begin{psmatrix}
	0 & 2 & 0 \\
	0 & 0 & 2 \\
	2 & 0 & 0
\end{psmatrix} \right\rbrace,
P = \begin{psmatrix}
	0 & 1 & 0 \\
	0 & 0 & 1 \\
	1 & 0 & 0
\end{psmatrix},d =4; \\
M \in \left\lbrace
\begin{psmatrix}
	2 & 1 & 1 \\
	1 & 1 & 0 \\
	1 & 0 & 1
\end{psmatrix},
\begin{psmatrix}
	2 & 1 & 1 \\
	1 & 0 & 1 \\
	1 & 1 & 0
\end{psmatrix},
\begin{psmatrix}
	1 & 1 & 1 \\
	1 & 0 & 2 \\
	1 & 2 & 0
\end{psmatrix},
\begin{psmatrix}
	1 & 1 & 1 \\
	1 & 1 & 1 \\
	1 & 1 & 1
\end{psmatrix} \right\rbrace,
P = \begin{psmatrix}
	1 & 0 & 0 \\
	0 & 0 & 1 \\
	0 & 1 & 0
\end{psmatrix}, d = 3; \\
M \in \left\lbrace
\begin{psmatrix}
	1 & 1 & 1 \\
	1 & 0 & 0 \\
	1 & 0 & 0
\end{psmatrix},
\begin{psmatrix}
	0 & 1 & 1 \\
	1 & 1 & 0 \\
	1 & 0 & 1
\end{psmatrix}^*,
\begin{psmatrix}
	0 & 1 & 1 \\
	1 & 0 & 1 \\
	1 & 1 & 0
\end{psmatrix}^* \right\rbrace,
P = \begin{psmatrix}
	1 & 0 & 0 \\
	0 & 0 & 1 \\
	0 & 1 & 0
\end{psmatrix}, d = 4.
\end{gather*}
Moreover, each type above, except possibly the last two starred ones, occurs for an algebra satisfying Hypothesis \ref{hyp.scy}.
\end{theorem}
\noindent Similarly as for the two-vertex theorem, most of the types arise (up to Morita equivalence)
as skew group algebras of the form $B \# \mb{Z}_3$ or $B \# S_3$ where $B$ is Artin-Schelter regular of dimension $3$,
though the appropriate $B$ is less obvious in some cases.  A few we obtain instead from Ore extensions of twisted Calabi-Yau algebras of dimension $2$.  All of the necessary conditions coming from Theorem~\ref{thm.hilb} hold for the last two starred types; we have been unable to find examples of these types, though we suspect they do occur.
Our computational tools in GAP and Maple suggest that the Hilbert series of a derivation-quotient algebra on one of these quivers with a random superpotential seems to match the predicative Hilbert series in Theorem \ref{thm.hilb}.

In the two vertex case or three vertex case with $P \neq I$, one can show that in an allowable type $(M, P, d)$ the matrix $M$ is normal, and so Theorem~\ref{thm.hilb} allows one to restrict the types based on the spectral radius of $M$.
This is no longer necessarily true in the three vertex case with $P = I$, and so this is most difficult case in the classification.
In fact there are types $(M, I, d)$ in this case where the polynomial $p(t)$ from Theorem~\ref{thm.hilb} has only roots of unity for zeros and $M$ is not normal.  
In some cases these quivers support twisted Calabi-Yau algebras, but in other cases they do not.

The first step in the analysis of the remaining case is an appeal to the Golod-Shaferevich inequality in order to bound the number of loops at each vertex:
\begin{proposition}[Proposition \ref{prop.loops}]
Let $A = \kk Q/(\partial_a \omega : a \in Q_1)$ be a derivation-quotient algebra for some finite quiver $Q$ and some $\mu$-twisted superpotential $\omega$ of degree $d$.  
If $\GKdim(A) < \infty$ and the vertex $v$ of $Q$ is fixed by $\mu$, then the number of loops at $v$ in $Q$ is less than or equal to $\ell(d)$, where $\ell(3) = 3$, $\ell(4) = \ell(5) = 2$, and $\ell(d) = 1$ for $d \geq 6$.
\end{proposition}


The limitation on the number of loops gives us a finite number of cases to consider corresponding to the diagonal entries in the
matrix $M$ of a possible type $(M, I, d)$.  This allows for a computer analysis of the types $(M, I, d)$ such that $\det p(t)$ has only roots of unity for zeros,
using that there are finitely many products of cyclotomic polynomials that $\det p(t)$ can be.  The resulting matrices $M$ fall into finitely many families,
some of which can be ruled out since they lead to a $p(t)^{-1}$ with nonpositive coefficients, or happen to have $M$ normal but with the wrong spectral radius.  A number of other possibilities are ruled out by showing that any derivation-quotient algebra on that quiver leads to an algebra of the wrong Hilbert series, by considering the beginning of a Gr\"{o}bner basis.
As an example of the type of quiver that looks good at first,
let $A \iso \kk Q/(\partial_a \omega : a \in Q_1, \deg(\omega)=3)$ with $Q$ as below.
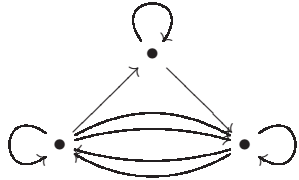
\vspace{1.5em}
\[  \xymatrix{
& & \bullet \ar[dr] \ar@(ul,ur)[] & &  \\
& \bullet \ar[ur] \ar@/^/[rr] \ar@/^1pc/[rr] \ar@(ul,dl)[] &
& \bullet \ar@/^/[ll] \ar@/^1pc/[ll] \ar@(ur,dr)[] & &}\]
\vspace{1em}

\noindent The roots of $\det p(t)$ all lie on the unit circle, and
computing the the first few terms of $h_A(t) = p(t)^{-1}$ suggests polynomial growth.
However, we can show that $Q$ has no superpotential such that the corresponding derivation-quotient algebra $A$ is twisted Calabi-Yau
(see Lemma~\ref{lem.partition} and Corollary~\ref{cor.partition}).

For the types $(M, I, d)$ remaining on the list, most are again realized by skew group algebras of actions on Artin Schelter regular algebras, or by
other techniques we have already used earlier.  However, for the first time we also encounter an infinite family of examples with non-normal $M$,
which can be realized by using quiver mutations.

Our final main result then follows.

\begin{theorem}[Proposition~\ref{prop.3vert3}]
\label{thm.3vert}
Let $A$ satisfy Hypothesis \ref{hyp.scy} with $|Q_0| = 3$ and type $(M, I, d)$.
Then the type is one of the following.
The indeterminates $a,b,c$ are positive integers satisfying Markov's equation $a^2+b^2+c^2=abc$.
\begin{gather*}
M \in \left\lbrace \begin{psmatrix}2 & 1 & 1 \\ 1 & 0 & 1 \\ 1 & 1 & 0\end{psmatrix},
\begin{psmatrix}2 & 1 & 1 \\ 1 & 1 & 0 \\ 1 & 0 & 1\end{psmatrix},
\begin{psmatrix}1 & 1 & 1 \\ 1 & 1 & 1 \\ 1 & 1 & 1\end{psmatrix},
\begin{psmatrix}1 & 1 & 1 \\ 1 & 2 & 0 \\ 1 & 0 & 2\end{psmatrix},
\begin{psmatrix}0 & a & 0 \\ 0 & 0 & b \\ c & 0 & 0\end{psmatrix} \right\rbrace,    d = 3; \\
M \in \left\lbrace \begin{psmatrix}1 & 1 & 0 \\ 0 & 1 & 1 \\ 1 & 0 & 1\end{psmatrix},
\begin{psmatrix}1 & 1 & 1 \\ 1 & 0 & 0 \\ 1 & 0 & 0\end{psmatrix},
\begin{psmatrix}0 & 1 & 1 \\ 1 & 0 & 1 \\ 1 & 1 & 0\end{psmatrix},
\begin{psmatrix}1 & 0 & 1 \\ 0 & 1 & 1 \\ 1 & 1 & 0\end{psmatrix},
\begin{psmatrix}1 & 1 & 0 \\ 0 & 1 & 1 \\ 2 & 0 & 1\end{psmatrix}^* \right\rbrace,  d = 4.
\end{gather*}
Moreover, each type, except possibly the final starred type, supports an algebra satisfying Hypothesis \ref{hyp.scy}.
\end{theorem}

One of the methods we employ is a GAP program which takes a random superpotential on the quiver and finds the beginning of a Gr{\"o}bner basis for the algebra, checking whether the resulting Hilbert series is as expected.  Applying this to the final starred type suggests that in degree $6$ the Hilbert series becomes wrong, so we are quite certain that this type does not occur.  In fact, we are able to show that the starred example cannot have an \emph{untwisted} superpotential for which the derivation-quotient algebra is twisted Calabi-Yau, which provides further strong evidence.

We view our classification as a test project to better understand how graded twisted Calabi-Yau algebras of finite GK dimension can arise.
We use only a few methods to obtain the types in this paper, primarily skew group algebras, Ore extensions, and quiver mutation, and we wonder if there is a short list of techniques like this which can produce all possible types.  Before we began this project we did not even think about the possibility of mutations, and in fact wrongly suspected that in any type $(M, P, d)$ of a twisted Calabi-Yau algebra the matrix $M$ would be normal.  Perhaps studying quivers with more vertices will produce examples which require new kinds of constructions.

\section*{Acknowledgments}
We thank Ellen Kirkman, Frank Moore and Manny Reyes for helpful conversations and support.
We also thank the anonymous referee for pointing out several points for clarification and simplification.

\section{Methods of constructing twisted Calabi-Yau algebras}
\label{sec.motive}

In this section, we discuss various methods that can be used to produce twisted Calabi-Yau algebras which occur as
factor rings of path algebras of quivers.

\subsection*{Hopf actions on Artin-Schelter regular algebras}
Recall that a graded $\kk$-algebra $R = \bigoplus_{n \geq 0} R_n$ which is connected (that is, $R_0 = \kk$) is
called Artin-Schelter (or AS) regular of dimension $d$ if it has graded global dimension $d$, and $R$ satisfies
the AS Gorenstein condition
\[
\Ext^i_R(\kk, R) \cong \begin{cases} \kk & i = d \\ 0 &  i \neq d. \end{cases}
\]
Here $\kk = R/R_{\geq 1}$ is the \emph{trivial module}.  Note that in this version of the definition of AS regular, we do
not include the assumption that $\GKdim(R) < \infty$ as in the original papers on the subject.
By \cite[Lemma 1.2]{RRZ}, an AS regular algebra $R$ is the same as a connected $\mb{N}$-graded twisted Calabi-Yau
algebra.   Thus if $R$ is AS regular of dimension $3$ and generated as an algebra by degree $1$ elements, it is a derivation quotient
algebra $\kk Q/(\partial_a \omega : a \in Q_1)$ for a twisted superpotential $\omega$ on a quiver $Q$ with one vertex, by Theorem~\ref{thm.bsw}.

Now let $H$ be a semisimple, finite-dimensional Hopf algebra with coproduct
$\Delta$ and counit $\epsilon$.  We use the Sweedler notation $\Delta(h) = h_1 \otimes h_2$.  Suppose that $H$
acts on the AS regular algebra $R$ of dimension $d$, so that $R$ is a left $H$-module algebra (that is, $R$ is a left $H$-module such that $h(rs) = h_1(r) h_2(s)$ and
$h(1) = \epsilon(h)$ for all $h \in H$, $r, s \in R$).   We assume that $H$ preserves the grading on $R$, in other words if $r \in R_n$
then $h(r) \in R_n$ for all $h \in H$.  The smash product $R \# H$ is defined to
be the tensor product $R \otimes_\kk H$ as a $\kk$-vector space, with multiplication $(r \otimes g) * (s \otimes h) = rg_1(s) \otimes g_2h$.
By \cite[Theorem 4.1]{RRZ}, the smash product $A = R\# H$ is again a graded twisted Calabi-Yau algebra of dimension $d$.  Moreover, its Nakayama automorphism $\mu_A$ is given by the formula $\mu_A = \mu_R \otimes (\Xi^l_{\hdet} \circ \mu_H)$.  Here, $\mu_R$ is the Nakayama automorphism of $R$,
and $\mu_H$ is the Nakayama automorphism of the algebra $H$.  The action of $H$ on $R$ determines an
algebra homomorphism $\hdet: H \to \kk$ called the \emph{homological determinant}, which is defined in \cite{JZ,KKZ}.  Then
$\Xi^l_{\hdet}: H \to H$ is the \emph{left winding automorphism} of $H$ with respect to $\hdet$, which is defined by
$\Xi^l_{\hdet}(h) = \hdet(h_1) h_2$.

Smash products of the type described above are presumably quite useful to construct more general twisted Calabi-Yau algebras, but finding
suitable Hopf actions and calculating $\hdet$ can be difficult.  In this paper we will only need the construction in the special case that $H=\kk G$ is a group algebra for a finite group $G$, where $\kk G$ is a Hopf algebra as usual with $\Delta(g) = g \otimes g$ and $\epsilon(g) = 1$ for all $g \in G$.
In this case $R \# \kk G$ is a \emph{skew group algebra} and the constructions above simplify considerably, as the left $\kk G$-module algebra structure is simply an action of $G$ on $R$ by graded automorphisms.
Recall our standing assumption that $\kk$ is algebraically closed of characteristic $0$.
Then $\kk G$ is semisimple and its Nakayama automorphism $\mu_G$ is the identity.
The winding automorphism $\Xi^l_{\hdet}(g) = \hdet(g) g$ simply scales group elements by their homological determinants.

Assume now that $R$ is generated as an algebra by $R_1$.  The algebra $A = R \# \kk G$ is also graded twisted Calabi-Yau of dimension $d$ and generated as an algebra by its elements of degree $0$ and $1$, and its
degree $0$ piece $A_0 = \kk G$ is semisimple.  As is well-known, we can choose a full idempotent $e \in A_0$ such that
$B = e A e$ is a graded \emph{elementary} algebra with $B_0 = \kk^m$ for some $m$, and where $B$ is still twisted Calabi-Yau
since it is Morita equivalent to $A$ \cite[Theorem 4.3, Lemma 6.2]{RR1}.  Now it is standard that $B \cong \kk Q/I$ for some uniquely determined quiver $Q$ and homogeneous ideal $I \subseteq \kk Q_{\geq 2}$ \cite[Lemma 3.4]{RR1}.  In fact the quiver $Q$ is the \emph{McKay quiver} of the action of $G$ on $R$ (see the discussion following \cite[Theorem 3.2]{BSW}).  Namely, let $V = R_1$, which is a finite-dimensional representation of $G$.  If $W_1, \dots, W_m$ are the distinct finite-dimensional irreducible representations of $G$, then $Q$ has $m$ vertices, labeled by the $W_i$,
and the number of arrows from vertex $i$ to vertex $j$ is the number of copies of
$W_i$ in the direct sum decomposition of $V \otimes W_j$ (note that our definition gives the opposite of the McKay quiver defined in \cite{BSW}, which is necessary since we compose paths in the quiver from left to right).

The action of the Nakayama automorphism $\mu_A$ on $A_0 \cong \kk G$ is given, by the earlier formula, by the winding automorphism $\Xi^l_{\hdet}(g) = \hdet(g) g$.  Assume now that $R$ is a derivation-quotient algebra, which is automatic if $\dim R = 3$ by Theorem~\ref{thm.bsw},
or more generally as long as $R$ is $m$-Koszul for some $m$.   Write $R = \kk \langle x_1, \dots, x_r \rangle/ (\partial_{x_i} \omega : 1 \leq i \leq r )$ for a twisted superpotential $\omega$ of some degree $s$.  Given any graded automorphism $\sigma$ of $R$, since $R$ is generated in degree $1$, we can lift $\sigma$ uniquely to an automorphism $\sigma$ of $\kk \langle x_1, \dots, x_r \rangle$, and then  $\sigma(\omega) = \hdet(\sigma) \omega$ by \cite[Theorem 3.3]{MS}.  This gives a way to explicitly calculate $\hdet$, by checking how automorphisms scale the superpotential.

One can then determine the permutation the Nakayama automorphism $\mu_B$ gives on the
vertices of $Q$.  To do this, decompose $1 \in \kk G$ as a sum of central idempotents $1 = e_1 + \dots + e_m$, one for each irreducible representation of $G$.  If the representation $W_i$ has character $\chi_i$, recall the explicit formula $e_i = \frac{\dim_k(W_i)}{|G|} \sum_{g \in G} \chi_i(g^{-1}) g$.  Assuming that one has calculated $\hdet$, it is easy to see how $\Xi^l_{\hdet}$ permutes these idempotents, and this
determines a permutation of the vertices of $Q$ via the correspondence $W_i \leftrightarrow e_i$.

In particular, in the special case that $\dim R = 3$, the results above show how to find the type $(M, P, s)$ of the twisted Calabi-Yau
algebra $B \cong \kk Q/I$ which is Morita equivalent to $R \# \kk G$.

\subsection{Normal incidence matrices of McKay quivers}

Note that the incidence matrix $M$ of a McKay quiver of a group action is always a normal matrix \cite[Proposition 2]{butin}. In the more general situation of a smash product $R \# H$ for a semisimple Hopf algebra $H$, the underlying quiver can also be described by the same McKay quiver construction \cite[Section 7]{CKWZ}.  We briefly discuss this more general situation, though it is not needed in our later classification. In some cases the incidence matrix of this kind of quiver must again be normal.

\begin{proposition}
Let $H$ be a finite-dimensional semisimple triangular Hopf algebra and $\gamma$ a representation of $H$.
The McKay quiver associated to $H$ and $\gamma$ has incidence matrix $M$ which is normal.
\end{proposition}
\begin{proof}
Because $H$ is triangular, $H$-MOD is a symmetric monoidal category \cite{M}.
Following \cite{butin}, let $\{\gamma_1,\cdots,\gamma_\ell\}$ be the
set of equivalence classes of irreducible $H$-representations.
Let $\chi_i$ be the character of $\gamma_i$ and $\chi$ the character of $\gamma$.
Define the adjacency matrix $M$ by
\[ \gamma_k \tensor \gamma = \bigoplus_{i=1}^\ell m_{ik} \gamma_i.\]
Similarly, define $N$ by
\[ \gamma_k \tensor \gamma^* = \bigoplus_{i=1}^\ell n_{ik} \gamma_i.\]
Since $H$ is finite-dimensional and semisimple, then $M=N^T$ by \cite[Corollary 2.3.4]{BBKNZ}.


Because $H$ is triangular, we have
\[ (\gamma_k \tensor \gamma) \tensor \gamma^*
= (\gamma \tensor \gamma_k) \tensor \gamma^*
= \gamma \tensor (\gamma_k \tensor \gamma^*).\]
Computing each we find,
\begin{align*}
(\gamma_k \tensor \gamma) \tensor \gamma^*
&= \left(\bigoplus_{i=1}^\ell m_{ik} \gamma_i\right) \tensor \gamma^*
= \bigoplus_{i=1}^\ell m_{ik} \left( \gamma_i \tensor \gamma^*\right) \\
&= \bigoplus_{i=1}^\ell m_{ik} \left( \bigoplus_{j=1}^\ell n_{ji} \gamma_j \right)
= \bigoplus_{i=1}^\ell \left( \sum_{j=1}^\ell m_{ik} n_{ji} \right) \gamma_j
\end{align*}
and
\begin{align*}
\gamma \tensor (\gamma_k \tensor \gamma^*)
&= \gamma \tensor \left( \bigoplus_{i=1}^\ell n_{ik} \gamma_i \right)
= \bigoplus_{i=1}^\ell n_{ik} \left( \gamma \tensor \gamma_i \right) \\
&= \bigoplus_{i=1}^\ell n_{ik} \left( \bigoplus_{j=1}^\ell m_{ji} \gamma_j \right)
= \bigoplus_{i=1}^\ell \left( \sum_{j=1}^\ell n_{ik} m_{ji} \right) \gamma_j.
\end{align*}
Hence, $MN=NM$. Since $M=N^T$, then $MM^T=M^TM$.
\end{proof}
Triangular Hopf algebras as above  were classified by Gelaki \cite{G}.  In particular, they are all obtained from a group algebra $\kk G$
of a finite group $G$ with a (potentially) twisted comultiplication.
More generally, it may be enough for the proof to go through if there is an isomorphism between
$V \tensor W$ and $W \tensor V$ in $H$-MOD.
That is, if $H$-MOD is a braided monoidal category, implying that $H$ is a quasi-triangular Hopf algebra.
However, the classification of such algebras is still open.

We are not sure if the incidence matrices of the quivers one gets for the smash products $R \# H$ for arbitrary
finite-dimensional Hopf algebras $H$ are always normal, but in any case the results above suggest that to find twisted Calabi-Yau algebras of
 types $(M, P, d)$ where $M$ is not normal, one is likely to need constructions other than smash products.
One of these is the process of mutation, as we describe later in this section.

\subsection*{An example}
Let $G$ be a finite group acting by graded automorphisms on the Artin-Schelter regular algebra $R$, and
let $A = R \# \kk G$ be the skew group algebra.  As we saw above, $A$ is Morita equivalent to a twisted Calabi-Yau algebra
$A' = \kk Q/I$, where $Q$ is McKay quiver of the action.  In some cases $Q$ is disconnected.  In this case, $A'$ decomposes
as a product of algebras $A' \cong \prod_{i=1}^r \kk Q_i/I_i$, where each $Q_i$ is a connected component of $Q$, and where
each algebra $\kk Q_i/I_i$ is still twisted Calabi-Yau \cite[Proposition 4.6]{RR2}.  Some important quivers arise in this way as connected components of McKay quivers rather than as full McKay quivers of group actions.

\begin{example}[{\cite[Example 8.4]{CHI}}]
\label{ex.component}
Let $G$ be the group with presentation
\[
G = \langle \sigma, \tau,  \lambda  | \sigma^6 = \lambda^6 = \tau^2 = 1, \lambda\sigma = \sigma \lambda, \lambda \tau = \tau \lambda, \lambda \sigma \tau = \tau \sigma^{-1} \rangle.
\]
Let $\zeta$ be a fixed primitive $6$th root of $1$ and let $G$ act by graded automorphisms on a graded algebra $R = \kk \langle x, y \rangle/I$,
where $I$ is homogeneous with $I \subseteq \kk \langle x, y \rangle_{\geq 2}$.  Assume that $V = \kk x + \kk y$ is the $2$-dimensional representation of $G$ given by
\[
\sigma \mapsto \begin{pmatrix} \zeta & 0 \\ 0 & \zeta^{-1} \end{pmatrix}, \ \ \tau \mapsto \begin{pmatrix} 0 & 1 \\ 1 & 0 \end{pmatrix}, \ \ \lambda \mapsto \begin{pmatrix} 1 & 0 \\ 0 & 1 \end{pmatrix}.
\]
For each $i \in \{0, 1, \dots, 5 \}$ we have a $2$-dimensional irreducible representation of $G$, $W_i$, given by
\[
\sigma \mapsto \begin{pmatrix} \zeta^i & 0 \\ 0 & \zeta^{-i-1} \end{pmatrix}, \ \ \tau \mapsto \begin{pmatrix} 0 & 1 \\ 1 & 0 \end{pmatrix}, \ \ \lambda \mapsto \begin{pmatrix} \zeta & 0 \\ 0 & \zeta \end{pmatrix}.
\]
One may check that $W_i$ and $W_{5-i}$ are isomorphic representations via the coordinate switch. On the other hand, the eigenvalues of $\sigma$ for each $i=0,1,2$ differ and so $W_0, W_1, W_2$ are distinct up to isomorphism.

Now since $\lambda$ is central in $G$ of order $6$, any finite-dimensional representation $X$ of $G$ breaks up as $X = X_0 \oplus \dots \oplus X_5$, where $X_i$ is a representation of $G$ on which $\lambda$ acts by scalar multiplication by $\zeta^i$.  Since $\lambda$ acts trivially in the representation $V$, it follows from this that the McKay quiver $Q$ of the action breaks up as a disjoint union of quivers $Q_i$, where the vertices in $Q_i$ are the irreducible representations where $\lambda$ acts by $\zeta^i$.

Now one may check that $W_0, W_1, W_2$ are the only irreducible representations of $G$ up to isomorphism where $\lambda$ acts by $\zeta$.
Moreover, it is straightforward to calculate that $V \otimes W_i \cong W_{i-1} \oplus W_{i+1}$, with indices modulo $6$. Given the isomorphisms among the $W_i$, this implies that the component $Q_1$ of $Q$ is the quiver $Q'$ given by
\begin{equation}
\label{eq.specialquiv}
  \xymatrix{
{}_2\bullet \ar@/^/[rr] \ar@(ul,dl)[] & &
{}_1\bullet \ar@/_/[rr] \ar@/^/[ll] & &
\bullet_3 \ar@(ur,dr)[] \ar@/_/[ll]}
\end{equation}
with incidence matrix $M = \begin{psmatrix}
	0 & 1 & 1 \\
	1 & 1 & 0 \\
	1 & 0 & 1
\end{psmatrix}$.

Suppose now that $R = \kk \langle x, y \rangle/I$ is Artin-Schelter regular of dimension $n$. In the special case that $R = \kk[x,y]$, then $R$ is AS regular of dimension $2$.
Assume that $G$ acts on $R$ such that the action of $G$ on $R_1$ is given by the representation $V$.  In this way, we get that $R \# \kk G$ is Morita equivalent to an algebra that is a product of algebras, one of which is of the form $\kk Q'/I'$. Thus we obtain a twisted Calabi-Yau algebra of dimension $n$ on the quiver $Q'$.  
On the other hand, it is easy to see that there is no skew group algebra itself $R \# \kk G$ which is Morita equivalent to $\kk Q'/J$; this would require $Q'$ to be a full McKay quiver of an action; since there are three vertices, the
group $G$ acting would have three irreducible representations and thus would require $G \cong \mb{Z}_3$ or $G \cong S_3$. But $Q'$ does not occur from an action of those groups on a two-dimensional representation.
This can be proved by a simple group-theoretic argument, as we thank the referee for pointing out. Since we will enumerate the McKay quivers associated to these groups later in any case, we simply point the reader to the Appendix to see that Q' does not occur.

We have only discussed derivation-quotient algebras above in the context of dimension $3$. More generally, one may obtain a an $N$-Koszul Calabi-Yau algebra of dimension $n$ by taking all mixed partials of order $\frac{N(n-3)}{2}+1$ of a twisted superpotential $\omega$ of degree $\frac{N(n-1)}{2}+1$ \cite[Theorem 6.8]{BSW}. In particular, when $n = 2$, the potential itself is the relation; so $yx-xy$ is the superpotential that gives $R$.  Now the method of calculating $\hdet$ described earlier in this section applies to this superpotential.  The action of $G$ described above leaves $yx-xy$ invariant, so $\hdet$ is trivial.  Hence the action of the Nakayama automorphism of $R \# \kk G$ on the vertices of the quiver $Q$ is trivial.  So the same is true of the Nakayama automorphism of $B \cong \kk Q'/I'$ which is Morita equivalent to a component of the product.

As we will use Proposition~\ref{prop.3vert2} below, there is an automorphism $\rho$ of $\kk G$ given by 
$\sigma \mapsto - \sigma$, $\tau \mapsto \tau$, $\lambda \mapsto \lambda$.
Then $R \# \kk G$ has an automorphism $1 \# \rho$ which acts on the McKay quiver and fixes the components $Q_i$; on $Q_1$ this automorphism is easily seen to interchange the irreducible representations $W_0, W_2$ and to fix $W_1$. Thus we get an induced order $2$ automorphism of $B$ which acts on the vertices of $Q'$ by switching the two vertices with loops.
\end{example}

\subsection*{Ore extensions of twisted Calabi-Yau algebras}
Assume that $R$ is a twisted Calabi-Yau algebra of dimension $d$ with Nakayama automorphism $\mu_R$.
Let $\sigma \in \Aut_\kk(R)$ be an algebra automorphism of $R$ and $\delta$ a $\sigma$-derivation of $R$. That is, $\delta(rs) = \sigma(r)\delta(s) + \delta(r)$ for all $r,s \in R$. 
The \emph{Ore extension} $A=R[t;\sigma,\delta]$ of $R$ is generated as an algebra over $R$ by $t$ with relations $tr=\sigma(r)t+\delta(r)$ for all $r \in R$.
By \cite[Theorem 3.3]{LWW}, $A$ is again twisted Calabi-Yau (of dimension $d+1$) with Nakayama automorphism $\mu_A$ given by $\mu_A(x)=\sigma \inv \circ  \mu_R(x)$ for $x \in R$ and
$\mu_A(t) = ut+b$, $u,b \in R$, with $u$ a unit.

Now suppose that $R \cong \kk Q/I$ as graded algebras, for some connected quiver $Q$ and homogeneous ideal $I \subseteq \kk Q_{\geq 2}$.  Thus $R_0 = \kk^m$ where $m$ is number of vertices in $Q$, and we write $R_0 = \kk e_1 + \dots + \kk e_m$ where the $e_i$ are the trivial paths.   
Then the Nakayama automorphism $\mu_R$ is a graded automorphism \cite[Proposition 5.2(5)]{RR1} and so permutes the idempotents $e_i$; as usual we write $\mu$ for the corresponding permutation of $\{1, \dots, n \}$ so that $\mu(e_i) = e_{\mu(i)}$.  
Let $M$ be the incidence matrix of $Q$ and $P$ the permutation matrix of $\mu_R$ on the idempotents, where $P_{ij} = \delta_{\mu(i), j}$.  
Assume also $\sigma$ is a graded automorphism of $R$ and that 
$\delta$ is a $\sigma$-derivation of $R$ which is homogeneous of degree $1$, so that $\delta(r) \in R_{n+1}$ for $r \in R_n$.
Then $A$ is again $\NN$-graded with $\deg(t)=1$.

The graded automorphism $\sigma$ also permutes the idempotents, say corresponding to a matrix $P'$, with the same conventions, so
we write $\sigma(e_i) = e_{\sigma(i)}$ and $P'_{ij} = \delta_{\sigma(i), j}$.
Then $t e_i = \sigma(e_i) t = e_{\sigma(i)} t$, so that $t = \sum t_i$, where $t_i  = t e_i \in e_{\sigma(i)}  A_1 e_i$.  Since $A_0 = R_0 = \kk^m$, while a $\kk$-basis of $A_1$ is given by a basis of $R_1$ together with $t_1, \dots, t_m$, we can write $A$ as a factor of a path algebra $A \cong \kk Q'/J$, where the $t_i$ give new arrows from $\sigma(i) \to i$ for each $i$ \cite[Lemma 3.4]{RR1}.
So $Q'$ has incidence matrix $M + (P')^{-1}$.  The ideal $J$ is generated by the relations generating $I$ together with degree 2 relations of the form
$t_i r = \sigma(r) t_j + \delta(r)$ for all $i,j$ and arrows $r \in e_i R_1 e_j$.  The Nakayama automorphism of $A$ is given by the formula above.  In particular, $\mu_A(e_i) = \sigma^{-1}(\mu_R(e_i)) = e_{\sigma^{-1}(\mu(i))}$ and so the permutation matrix corresponding to the action of $\mu_A$ on the idempotents is  $P (P')^{-1}$.

Now suppose that $R \cong \kk Q/I$ is graded and twisted Calabi-Yau of dimension $2$, so $A$ is graded twisted Calabi-Yau of dimension $3$.  Then $I$ is generated by degree $2$ relations \cite[Proposition 7.1]{RR1}.  Thus $J$ is also generated by degree $2$ relations.  It follows
that $A$ is a derivation-quotient algebra corresponding to a degree $3$ superpotential, by Theorem~\ref{thm.bsw}.
By the calculations in the previous paragraph, $A$ has type $(M+(P')^{-1}, P (P')^{-1}, 3)$.

\subsection*{Quiver mutations}
One method for producing quivers supporting twisted Calabi-Yau algebras on quivers
whose adjacency matrices are not normal is via quiver mutation,
introduced originally in the context of cluster algebras by Fomin and Zelevinsky \cite{FZ1, FZ2}.
We have adapted these definitions to our convention of composing paths from left to right.

Let $Q$ be a quiver without loops or oriented 2-cycles and fix a vertex $v$.
For an arrow $a \in Q_1$ denote by $s(a)$ the source
and by $t(a)$ the target of $a$.  Consider the following process:
\begin{itemize}
\item[Step 1] For every pair of arrows $a,b \in Q_1$ with $t(a)=v$ and $s(b)=v$,
create a new arrow $[ab]:s(a) \rightarrow t(b)$.
\item[Step 2] Reverse each arrow $a$ with source or target at $v$ and rename it $a^*$.
\item[Step 3] Remove any maximal disjoint collection of oriented 2-cycles.
\end{itemize}
The resulting quiver $\widetilde{Q}$ is called a \textit{mutation} of $Q$.

Derksen, Weyman, and Zelevinsky \cite[Section 5]{dwz} gave a method of mutating a superpotential $\omega$ on a quiver $Q$
to produce a superpotential $\widetilde{\omega}$ on $\widetilde{Q}$,
defined as follows.   For each pair $a, b$ as in Step 1, replace every occurrence of the product $ab$ in $\omega$
by the single arrow $[ab]$, and add the term $[ab]b^* a^*$ to the potential.  Then when performing the reduction of Step 3,
the new potential also reduces in a standard way (we omit the details).  In the examples in the next proposition, where $\omega$ always has degree $3$, this last step will simply remove the 2-cycles from the new potential, leaving another homogeneous potential of degree $3$.  It is worth noting, however, that in general the mutation process does not lead to a homogeneous potential, and so the corresponding derivation-quotient algebra
may no longer be graded.

\begin{proposition}
\label{prop.mutation}
Let $Q$ be a quiver of the form
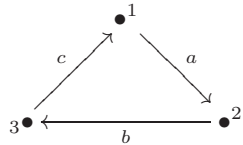
\[ \xymatrix{
	&	\bullet^1 \ar@{->}^a[dr] 	&  \\
{}_3\bullet \ar@{->}^c[ur] &  & \bullet^2 \ar@{->}^b[ll] }\]
where $a,b,c$ represent the number of arrows between each pair of vertices.
Then there is a superpotential $\omega$ of degree $3$ on $Q$ such that
$A \iso \kk Q/(\partial_a \omega : a \in Q_1)$ is Calabi-Yau with $\GKdim(A) = 3$
(and so $A$ satisfies Hypothesis~\ref{hyp.scy})
if and only if $a,b,c$ satisfy Markov's equation $a^2+b^2+c^2=abc$.
\end{proposition}

\begin{proof}
Suppose first that there is a (non-twisted) superpotential $\omega$ of degree such that
$A \iso \kk Q/(\partial_a \omega : a \in Q_1)$ is Calabi-Yau with $\GKdim(A) = 3$.
Let $M = \begin{psmatrix}0 & a & 0 \\ 0 & 0 & b \\ c & 0 & 0\end{psmatrix}$ be the incidence matrix of $Q$,
and let $p(t) = I - Mt + M^T t^2 + It^3$ be the associated matrix polynomial.  By Theorem~\ref{thm.hilb}
(see also the expanded version in Theorem~\ref{thm.hilb2} below), the matrix Hilbert series of $A$
is $h_A(t) = p(t)$, where all zeroes of $\det p(t)$ are roots of unity, with $1$ occurring as a zero with multiplicity at least $3$.
An easy calculation gives
\[ \det(p(t)) = 1+(-abc+a^2+b^2+c^2-3)t^3+(abc-a^2-b^2-c^2+3)t^6-t^9.\]
Now one checks that since $\frac{d}{dt} \det(p(t))$ also vanishes at $t = 1$, this gives $a^2 + b^2 + c^2 - abc = 0$.

Conversely, Iyama and Reiten showed that for every triple $(a, b, c)$ satisfying Markov's equation,
there is a superpotential $\omega$ on $Q$ such that $A \iso \kk Q/(\partial_a \omega : a \in Q_1)$ is Calabi-Yau of
dimension $3$ (see the discussion following \cite[Proposition 7.2]{IR}).
It then easily follows from Theorem~\ref{thm.hilb2}(2) below that $\GKdim(A) = 3$, by noting that
$\det p(t) = 1 - 3t^3 + 3t^6 - t^9 = (1-t^3)^3$ vanishes precisely three times at $t = 1$,
while the entries of $\adj p(t)$ do not vanish at $t = 1$.

More explicitly, Iyama and Reiten show that if one begins with the quiver $Q$ with $(a, b, c) = (3, 3, 3)$, then there is a Calabi-Yau algebra
$A$ on $Q$ obtained as a skew group algebra $\kk[x,y,z] \# \kk \mb{Z}_3$, where $\mb{Z}_3 = \langle \sigma \rangle$ and
$\sigma$ scales all three variables by $\zeta$, and so $Q$ is the McKay quiver of this action.   The algebra $A$ must be a derivation-quotient
algebra $\kk Q/(\partial_a \omega : a \in Q_1)$ for some superpotential $\omega$ by Theorem~\ref{thm.bsw}, but in any
case it is easy to write down $\omega$.  Performing a mutation on $Q$ and $\omega$, the derivation quotient algebra $\kk Q/(\partial_a \omega : a \in Q_1)$ gets mutated to another one $\kk \widetilde{Q}/(\partial_a \widetilde{\omega} : a \in Q_1)$.  Now beginning with the example $A$ above and performing all possible mutations, one gets Calabi-Yau algebras on all quivers $Q$ with $(a, b, c)$ a solution to Markov's equation.  The fact that all possible triples is obtained comes from the well-known fact that all solutions to Markov's equation can be obtained by starting with $(3, 3, 3)$ and performing all iterations of the operation $(a, b, c) \mapsto (a, b, ab-c)$, together with the analogous operations in the $a$ or $b$ coordinate.  One can easily check that mutating the quiver $Q$ corresponding to $(a, b, c)$ at the vertex $2$ leads to the quiver $Q$ corresponding to $(a, b, ab-c)$.  In fact, Iyama and Reiten prove that the mutated algebra is still Calabi-Yau by describing it in an alternate way as the endomorphism ring of a tilting module \cite[Theorem 7.1]{IR}.

For further illustration, consider the quiver $Q$ corresponding to $(a, b, c)$. Label the $a$ arrows from $1$ to $2$ by $x_1, \dots, x_a$, and similarly label the arrows from $2$ to $3$ by $y_i$ and the arrows from $3$ to $1$ by $z_i$.  Without loss of generality, assume we mutate at vertex $2$.  We note that in any solution to Markov's equation, $ab > c$.  Thus in mutating at vertex $2$, for example, the $ab$ new arrows created
of the form $[x_iy_j]$ from $1$ to $3$ form $c$ $2$-cycles with the $c$ arrows from vertex $3$ to vertex $1$.  Thus Step 3 of the mutation process removes $c$ $2$-cycles and all of the original $c$ arrrows from $3$ to $1$, leaving $ab-c$ new arrows from $1$ to $3$.  The mutated potential only contains terms of the form $[x_i y_j] y_j^* x_i^*$.
\end{proof}

\section{Key results}
\label{sec.key}

In this section, we describe some results that we will use to limit the type of a twisted Calabi-Yau
algebra of finite GK dimension.
First, we give a more detailed statement of Theorem~\ref{thm.hilb}, which follows from results in \cite{RR1}.
\begin{theorem}[{\cite[Propositions 2.8, 8.2, 8.10, Corollary 8.11]{RR1}}]
\label{thm.hilb2}
Let $A = \kk Q/(\partial_a \omega : a \in Q_1)$ be a derivation-quotient algebra on the connected quiver $Q$ with twisted superpotential $\omega$ of degree $d$.  Suppose that $A$ is twisted Calabi-Yau of type  $(M, P, d)$.
\begin{enumerate}
\item The matrices $M$ and $P$ commute, and we have $h_A(t) = (p(t))^{-1}$, where
\begin{align}
\label{def.mpoly}
	p(t) = I- M t+ PM^T t^{d-1}- P t^d \in M_m(\kk)[t].
\end{align}
In particular, writing $(p(t))^{-1} = \sum_{n \geq 0} H_n t^n$ then each $H_n$ has nonnegative entries.
\item Every zero of $\det p(t) \in \kk[t]$ is a root of unity if and only if $\GKdim(A) < \infty$.
In this case, writing $(p(t))^{-1} = \adj(p(t)) \det(p(t))^{-1}$,
if $m_d$ is the multiplicity of vanishing of $\det(p(t))$ at $t = 1$,
and $m_a$ is the maximal multiplicity of vanishing of the matrix entries of $\adj p(t)$ at $t = 1$,
then $\GKdim(A) = m_d - m_a$.  In particular, if $3 \leq \GKdim(A) < \infty$ then $m_d \geq 3$.
\item Suppose that $M$ is normal.  If $3 \leq \GKdim(A) < \infty$, then $\GKdim(A) = 3$, either $d = 3$ or $d = 4$,
and the spectral radius $\rho(M)$ of $M$ satisfies $\rho(M) = 6-d$.
\item If $M$ is symmetric, $P = I$, and $d = 3$, then all eigenvalues of $M$ must lie in the interval $[-1, 3]$.
\end{enumerate}
\end{theorem}

Given a twisted Calabi-Yau derivation-quotient algebra of type $(M, P, d)$, the more complicated $P$ is, the more
it restricts $M$ since $M$ and $P$ commute.   This will make the case $P = I$ the hardest to deal with below.
Recall that an $m \times m$ matrix $M$ is \emph{circulant} if $M_{ij} = M_{i+1, j+1}$ for all $i,j$, where indices are
taken modulo $m$.  The following proposition follows from direct calculation and is left to the reader.
\begin{proposition}
\label{prop.perms}
Let $M, P \in M_m(\mb{Z})$ be commuting matrices.  Assume that $P$ is a permutation matrix corresponding to
the permutation $\sigma \in S_m$.
\begin{enumerate}
\item If $\sigma=(12\cdots m)$, then $M$ is circulant.
\item Suppose that $\sigma=(23\cdots m)$, so
we can write $P$ in the block form $P = \begin{pmatrix}1 & 0 \\ 0 & R\end{pmatrix}$
where $R$ is an $(m-1) \times (m-1)$ permutation matrix.   Then writing $M$
in the corresponding block form $M = \begin{pmatrix}a & \mathbf{b} \\ \mathbf{c}^T & N\end{pmatrix}$
for some row vectors $\mathbf{b}, \mathbf{c} \in k^{m-1}$ and $(m-1) \times (m-1)$-matrix $N$,
there are constants $b, c$ such that $\mathbf{b}=(b,\cdots,b)$, $\mathbf{c}=(c,\hdots,c)$, and $N$ is circulant.
\end{enumerate}
\end{proposition}

In the remainder of this section we show how to use the standard ideas of
Golod-Shafarevich to give some rough restrictions on what types $(M, P, d)$ can
be associated to a twisted Calabi-Yau algebra $A$ of dimension $3$ with finite GK dimension.
In this case the Hilbert series of $A$ will have the form given in Theorem~\ref{thm.hilb2}.  For some
choices of type, just because of the degrees of the generators and relations,
an algebra with that Hilbert series would be forced to have exponential growth.

We will use this method primarily to limit the number of loops at a vertex in a quiver supporting
an twisted Calabi-Yau algebra of finite GK dimension.  We also note that the method
can be used  to rule out larger subquivers, which may be useful in the future.

We first recall the idea of the Golod-Shafarevich bound for factor rings of a free associative algebra, as in \cite{GS}.
Given two power series $f(t) = \sum a_n t^n$ and $g(t) = \sum b_n t^n$ with $a_n, b_n \in \mb{Q}$,
we write $f(t) \leq g(t)$ if $a_n \leq b_n$ for all $n \geq 0$.   If $A = \bigoplus_{n \geq 0} A_n$ is a graded algebra with $m$ (degree 1) generators and $r_i$ relations of degree $i$, then
the Hilbert series $h_A(t) = \sum_{n \geq 0} (\dim_{\kk} A_n) t^n$ satisfies
\begin{align}
\label{gs.ineq}
h_A(t) \geq \big|(1-mt+\sum_{i\geq 2} r_i t^i )^{-1}\big|.
\end{align}
Here, the absolute value of a power series is obtained by replacing the first
negative coefficient and all subsequent coefficients by zero.

In the next result, we give an analog of the Golod-Shafarevich bound  \eqref{gs.ineq}  for matrix-valued Hilbert series.  If $p(t) = \sum_{n \geq 0} H_nt^n$ is a matrix valued power series with $H_n \in M_m(\mb{Q})$, then $|p(t)|$ is obtained by replacing the first matrix $H_n$ such that $H_n$ does not have all nonnegative coefficients, and all $H_r$ with $r \geq n$, by the zero matrix.
Similarly, if $q(t) = \sum_{n \geq 0} G_n t^n$ is another matrix power series then we write $p(t) \leq q(t)$ if $G_n-H_n$ has nonnegative coefficients for all $n$.

\begin{proposition}
\label{prop.gs}
Let $Q$ be a finite quiver with $m$ vertices and let $A = \kk Q/I$,
where $I$  is generated by a set $X$ where each $x \in X$ satisfies $x \in e_i \kk Q_n e_j$ for some $i, j$ and $n \geq 2$.
Let $M$ be the incidence matrix of $Q$, and let $R = R(t)$ be the matrix polynomial $R = \sum_{n \geq 0} H_n t^n$ with $(H_n)_{i,j}$ equal to the number of elements in $X \cap e_i (\kk Q)_n e_j$. Then
\[ h_{A}(t) \geq \left| \left( I - Mt + R \right)^{-1}\right|.\]
\end{proposition}
\begin{proof}
This is similar to \cite[Theorem 2.3.4]{EE}.
As in the proof of that result, considering the projective resolution of $S = A/A_{\geq 1} \cong k^m$ as a graded $A$-module leads to the inequality $h_A(t)  (I - Mt + R) \geq I$.  See also \cite[Lemma 3.4]{RR1}.

We then argue the absolute values as in \cite[Proposition 2.3]{PP}.
Let $F(t), G(t)$ be matrix-valued power series, where $F(t) \geq 0$, $G(t)$ is invertible in $M_r(\mb{Q}\llbracket t\rrbracket)$, and $F(t)G(t) \geq I$.   We claim that $F(t) \geq |G(t)\inv|$.  Let $k$ be minimal such that $G(t)\inv$ has a negative entry in some coefficient at $t^k$ (if no such $k$ exists put $k = \infty$). Hence, the equation $F(t) \geq |G(t)\inv|$ holds trivially in degrees $\geq k$. In degrees $< k$, $|G(t)\inv| = G(t)\inv$, so $F(t) \geq |G(t)\inv|$ follows from the hypothesis $F(t)G(t) \geq I$.


Applying the claim with $F(t) = h_A(t)$ and $G(t) = (I - Mt + R)$ gives the desired result.
\end{proof}

Applying the previous proposition to the special case of algebras defined by a superpotential yields the following corollary. We first set up some notation.
Suppose $Q$ is a finite quiver with $m$ vertices, and suppose that $Q'$ is a full subquiver of $Q$.  In other words, $Q'$ consists of some
subset of vertices of $Q$ and all arrows in $Q$ which have both head and tail in that subset.  Assume by renumbering if necessary that the vertices in $Q'$ are the ones labeled $1, \dots, m'$, where $m' \leq m$.  Given an $m \times m$ matrix $N$, let $\widetilde{N}$ be the $m' \times m'$ submatrix given by taking the first $m'$ rows and columns.  We have an identification
$\kk Q' = \kk Q/J$, where $J$ is the ideal spanned by all paths (including trivial ones) involving a vertex in $Q \setminus Q'$.  Now given a graded algebra $A = \kk Q/I$, its \emph{restriction} to $Q'$ is the algebra $A' = \kk Q'/I'$, where $I' = (J + I)/J$.
\begin{corollary}
\label{cor.gs}
Let $Q$ be a finite quiver with incidence matrix $M$, and let $A = \kk Q/(\partial_a \omega : a \in Q_1)$ be a derivation-quotient algebra for some twisted superpotential $\omega$ of degree $d$.  Let $\mu$ be the corresponding permutation of the vertices of $Q$, so that
each path in $\omega$ goes from a vertex $i$ to the vertex $\mu(i)$.  Let $P$ be the corresponding permutation
matrix, where $P_{ij} = \delta_{\mu(i) j}$.  Let $Q'$ be a full subquiver of $Q$, and let $A' = \kk Q'/I'$ be the restriction as defined above.

Then we have the inequality of $m' \times m'$ matrix valued Hilbert series:
\[
h_{A'}(t) \geq \left| \left( I - \widetilde{M}t + \widetilde{PM^T}t^{d-1} \right)^{-1}\right|.
\]
\end{corollary}
\begin{proof}
First, $\widetilde{M}$ is clearly the incidence matrix of the full subquiver $Q'$.
If $a$ is an arrow from $i$ to $j$, any path contained in the superpotential that begins with $a$ has the form $ap$ where $p$ is a path of length $d-1$ from $j$ to $\mu(i)$; thus $\partial_a \omega$ is a relation from $j$ to $\mu(i)$. The ideal $I'$ is generated by those relations $\partial_a \omega$ that do not become $0$ in $\kk Q'$; in particular, any such relation must start and end at vertices in $Q'$.
If $k, \ell \in Q_0'$, then the number of relations of the form $\partial_a \omega$ from $k$ to $\ell$ is the same as
the number of arrows in $Q$ from $\mu^{-1}(\ell)$ to $k$, which is $M_{\mu^{-1}(\ell), k}$.   Thus
$(PM^T)_{k, \ell} = (M)_{\ell, \mu(k)} = M_{\mu^{-1}(\ell), k}$ since $PM = MP$ by Theorem~\ref{thm.hilb2}.
Thus $I'$ is generated by a set $R'$ of relations with corresponding weighted incidence matrix $\widetilde{P M^T} t^{d-1}$.
The result now follows from Proposition~\ref{prop.gs}.
\end{proof}

In our applications of the results above  we will need the following simple analytic lemma.
\begin{lemma}
\label{lem.recur}
Let $s \geq 2$ be an integer and $m \in \mathbb{R}$.  Consider $f(x) = x^s - m x^{s-1} + m$.
Define a sequence $\{r_n \}$ by the recurrence relation $r_1 = r_2 = \dots = r_{s-1} = m$, and
$\displaystyle r_n = m\left(1 - \frac{1}{r_{n-s+1} r_{n-s+2} \dots r_{n-1}}\right)$ for $n \geq s$.
Suppose that $\sqrt[s-1]{s} \frac{s}{s-1}  \leq m$.  Then $f(x)$ has a unique real root $u$
with $\frac{s-1}{s}m \leq u < m$.  Moreover $u < r_n \leq r_{n-1} \leq m$ holds for all $n \geq 2$.
\end{lemma}
\begin{proof}
A direct calculation shows that $f(\frac{(s-1) m}{s}) = - \frac{(s-1)^{s-1} m^s}{s^s} + m \leq 0$, since $\sqrt[s-1]{s} \frac{s}{s-1} \leq m$.
As $f(m) = m$, the intermediate value theorem
yields a real root $u$ of $f$ with $\frac{s-1}{s}m \leq u < m$.  Moreover, $f'(x)$ has only $\frac{s-1}{s}m$ and $0$ as roots, so $f$ must be increasing
on the interval from $\frac{s-1}{s}m$ to $m$ and so $u$ is unique.

The inequalities $u < r_n \leq  r_{n-1} \leq m$ are now proved for all $n \geq 1$ by induction.  They hold by definition for $2 \leq n \leq s-1$.
Assuming that $n \geq s$ and $u < r_i$ for all $i \leq n-1$, we have
\[
r_n = m \left(1 - \frac{1}{r_{n-s+1} \cdots r_{n-1}} \right) > m \left(1 - \frac{1}{u^{s-1}} \right) = u,
\]
proving that $u < r_n$.  Then assuming $r_i \leq r_{i-1} \leq m$ for $i \leq n-1$ we have
\[
r_n = m \left(1 - \frac{1}{r_{n-s+1} \cdots r_{n-1}} \right) \leq m \left(1 - \frac{1}{r_{n-s} \cdots r_{n-2}} \right) = r_{n-1} \leq m,
\]
completing the induction step.
\end{proof}

The main application in this paper of the Hilbert series bounds above will be to
limit the number of loops that can occur at a vertex in a twisted Calabi-Yau algebra with polynomial growth.
\begin{proposition}
\label{prop.loops}
Let $A = \kk Q/(\partial_a \omega : a \in Q_1)$ be a derivation-quotient algebra for some finite quiver $Q$ and some $\mu$-twisted superpotential $\omega$ of degree $d$.  
If $\GKdim(A) < \infty$ and the vertex $v$ of $Q$ is fixed by $\mu$, then the number of loops at $v$ in $Q$ is less than or equal to $\ell(d)$, where $\ell(3) = 3$, $\ell(4) = \ell(5) = 2$, and $\ell(d) = 1$ for $d \geq 6$.
\end{proposition}
\begin{proof}
Recall that $\omega$ must be $\mu^{-1}$-twisted, and so is a sum of paths, each from a vertex $i$ to a vertex $\mu(i)$.
We apply Corollary~\ref{cor.gs} to the full subquiver $Q'$ consisting of the vertex $v$ and all incident loops.
Assume the notation introduced before Corollary~\ref{cor.gs}, and relabel so that $v$ is the vertex labeled $1$.  Let $P$ be the matrix the permutation $\mu$ induces on the vertices.  Since the vertex $1$ is fixed by $\mu$, $\widetilde{PN} = \widetilde{N} = N_{1,1}$ for any matrix $N$, and so Corollary~\ref{cor.gs} gives the bound $h_{A'}(t) \geq | (1 - mt + mt^{d-1})^{-1} |$, where $m$ is the number of loops at $v$.

Let $s = d-1$.  Now $(1 -mt + mt^s)^{-1} = \sum_{n \geq 0} a_n t^i$, where $a_n = m a_{n-1} - m a_{n-s}$ for $n \geq s$, with initial conditions $a_0 = 1, a_1 = m, \dots, a_{s-1} = m^{s-1}$.  Defining $r_n = a_n/a_{n-1}$ for $n \geq 1$, $r_n$ satisfies the recurrence studied in Lemma~\ref{lem.recur}, namely $\displaystyle r_n = m\left(1 - \frac{1}{r_{n-s+1} r_{n-s+2} \dots r_{n-1}}\right)$ for $n \geq s$, with initial conditions $r_1 = r_2 = \dots = r_{s-1} = m$.

Suppose now that $m > \ell(d)$.  A straightforward case-by-case check shows that this implies $\sqrt[s-1]{s} \frac{s}{s-1}  \leq m$. Thus Lemma~\ref{lem.recur} applies and shows in particular that $r_n > 0$ for all $n \geq 1$; thus $a_n > 0$ for all $n \geq 1$ and $|(1-mt+mt^s)^{-1}| = (1-mt + mt^s)^{-1}$.
Finally, the roots of $(1-mt + mt^s)$ are the reciprocals of the roots of $(x^s - mx^{s-1} + m)$, and we saw in Lemma~\ref{lem.recur} that this polynomial has a real root $u$ with $\frac{s-1}{s}m \leq u$.  
Since $m$ is an integer and $m \geq \ell(d)+1$, then $u>1$. The condition $r_n > u$ follows from Lemma \ref{lem.recur} and so, by definition of the $r_n$, the $a_n$ grow exponentially with $n$.
Hence the coefficients of the series $h_{A'}(t)$ grow exponentially and the algebra $A'$ has exponential growth. Finally, since $A'$ is by definition a homomorphic image of $A$, this implies that $A$ has exponential growth, contradicting the hypothesis.  We conclude that
$m \leq \ell(d)$ as required.
\end{proof}

The same idea as in the proof above can also be used to rule out full subquivers on more than one vertex,
and some examples are provided in the next proposition.
This is sensitive in general, and the most difficult part of applying the bound is to prove that the resulting series has positive coefficients.
We do not use this result in our classification, and so we omit the proof.
\begin{proposition}
\label{prop.badblock}
Let $A =\kk Q/(\partial_a \omega : a \in Q_1)$ be a derivation-quotient algebra for a $\mu$-twisted superpotential $\omega$ of degree $3$, where $\mu$ fixes the vertices $v,w$ of $Q$. Let $Q'$ be the
full subquiver on the vertices $v, w$. If $\GKdim(A) < \infty$, then the incidence matrix $N$ of $Q'$ cannot be of either of the following types:
\begin{align*}
\begin{pmatrix}2 & y \\ y & 2\end{pmatrix} \; (y \geq 2)  \quad\text{or}\quad
\begin{pmatrix}3 & z \\ z & 3\end{pmatrix} \; (z \geq 3).
\end{align*}
\end{proposition}

\section{The two-vertex case}
\label{sec.2vert}

In this section we completely classify two-vertex quivers which support algebras satisfying Hypothesis \ref{hyp.scy}.
In this case, there are only two possibilities for the permutation matrix $P$ in the type, and we handle each separately.  Theorem \ref{thm.2vert} follows from Lemmas \ref{lem.2vert1} and \ref{lem.2vert2} below.

\begin{lemma}
\label{lem.2vert1}
Let $A \iso \kk Q/(\partial_a \omega : a \in Q_1, |Q_0|=2)$ satisfy Hypothesis~\ref{hyp.scy} with
type $\left(M, P = \left( \begin{matrix}0 & 1 \\ 1 & 0\end{matrix} \right), \deg(\omega)\right)$.
Then $M$ is one of the following:
\[
\deg(\omega)=3:
\begin{pmatrix}1 & 2 \\ 2 & 1\end{pmatrix},
\begin{pmatrix}2 & 1 \\ 1 & 2\end{pmatrix},
\begin{pmatrix}0 & 3 \\ 3 & 0\end{pmatrix}, \qquad
\deg(\omega)=4:
\begin{pmatrix}1 & 1 \\ 1 & 1\end{pmatrix}.
\]
Moreover, every such $M$ occurs for an algebra $A$ of the form $R \# \kk \mb{Z}_2$, where $R$ is an Artin-Schelter regular
algebra of dimension $3$.
\end{lemma}

\begin{proof}
Write $M=\begin{pmatrix}a&b\\c&d\end{pmatrix}$.
By Theorem~\ref{thm.hilb2}, $MP = PM$ and so $a=d$ and $b=c$.  Hence, $M$ is symmetric with eigenvalues $a+b$ and $a-b$.
By Theorem \ref{thm.hilb2}, $\deg(\omega) = 3$ or $4$, and $\rho(M)+\deg(\omega)=6$.
Hence, for a degree 3 superpotential, the only possibilities are $(a,b)=(1,2)$, $(2,1)$, or $(0,3)$, since
if $(a, b) = (3, 0)$ then the corresponding quiver is not connected.  These give the possibilities stated.

For a degree 4 twisted superpotential, $\rho(M)=2$ so  that $(a,b)=(1,1)$ or $(0,2)$, since
$(2, 0)$ is excluded again since it leads to a disconnected quiver.  When $(a, b) = (0, 2)$ we examine
the resulting matrix Hilbert series.   In this case we have
$M = 2P$ and thus $p(t) = I - Mt + PM^T t^3 - Pt^4 = I-2Pt + 2I t^3 - P t^4$ diagonalizes to
$U^{-1} p(t) U = \operatorname{diag}(1 - 2t + 2t^3 - t^4, 1 + 2t + 2t^3 + t^4)$, where
$U = \begin{pmatrix} 1 & 1 \\ 1 & -1 \end{pmatrix}$. Note that $1 + 2t + 2t^3 + t^4$ does not have all of its zeros on the unit circle, since it has a real root between $-1$ and $1$ by the intermediate value theorem.  Thus
$\det U^{-1} p(t) U = \det p(t)$ has a zero which is not a root of unity, contradicting Theorem~\ref{thm.hilb2}.
Again this leaves precisely the one possibility above.

The superpotentials $\omega$ corresponding to the regular algebras $R=T(V)/(\partial_a \omega)$,
and the actions of $\mb{Z}_2$ on $R$ such that $A = R \# \kk \mb{Z}_2$ has each of the listed
types, are given in Table~\ref{table.2vert1}.  We explain this table in detail here, and leave most of the similar verifications of future tables to the reader.

It should be clear from $\omega$ whether $V$ has basis $\{x,y,z\}$ or $\{x,y\}$.  First, one should check
that each of the twisted potentials $\omega$ is good, that is, that the corresponding derivation-quotient algebra $A = \kk Q/(\partial_a \omega : a \in Q_1)$ is AS regular.  For the first three lines in the table,
the corresponding algebra is of the form
\[
A = \kk \langle x, y, z \rangle/(yx-pxy, xz-qzx, zy-ryz),
\] where either $p = q= r = 1$ or $p = q = r = -1$, and such skew polynomial rings are well known to be regular of dimension $3$ for any nonzero $p, q, r$.  In the last line of the table, $\omega(H)$ is of type $H$ in Artin and Schelter's classification of families of  twisted superpotentials in \cite[Table 3.9]{AS}.  The corresponding algebra $A(H) = \kk Q/(\partial_a \omega : a \in Q_1)$ must be regular by \cite[Theorem 3.10]{AS}, since this potential is already generic in its family.

Next, one verifies that the given action of $\sigma$ on $V$ gives an automorphism of $A$.
For this it suffices to show that the automorphism $\sigma$ induces on the tensor algebra $T(V)$
satisfies $\sigma(\omega) = a \omega$ for some nonzero scalar $a$.  Once this is done, recall that in fact
$a = \hdet(\sigma)$ by \cite[Theorem 3.3]{MS}.  Then the Nakayama automorphism of $A$ in degree $0$
is given by the winding automorphism of $\kk G$ where $g \mapsto \hdet(g) g$ for each group element.
Finally, one calculates the corresponding  permutation of the vertices of the McKay quiver $Q$
by writing $1 \in \kk G$ as a sum of central idempotents (which correspond to the irreducible representations of $G$,  and hence the vertices of the quiver).   For each entry of Table~\ref{table.2vert1}, one has $\hdet(\sigma) = -1$,
so that the winding automorphism does indeed switch the idempotents $e_1 = (e + \sigma)/2$ and $e_2 = (e - \sigma)/2$, leading to the claimed permutation matrix $P$. Note that since the other tables below are sorted by the intended $P$, in each table the $\hdet$ of $G$ should be the same for all entries.

Finally, the incidence matrix $M$ of the McKay quiver
$Q$ is easily found from the definition of the McKay quiver, by first decomposing $V$ as a direct sum of irreducible
representations of $G$.  For the reader's convenience we have listed all of these McKay quivers in an appendix.
For example, in the first entry of Table~\ref{table.2vert1}, diagonalizing the action of $\sigma$ we see that in the notation of the appendix the character of $\sigma$ is the sum of the three characters $\chi_1 + \chi_2 + \chi_2$, and so $M$ can be read off of the $(1, 2, 2)$ entry in the corresponding table there.
 \end{proof}

\begin{table}
\caption[Table1]{$G = \ZZ_2 = \langle \sigma \rangle$ acting on $R=T(V)/(\partial_a \omega)$, type of $R \# \kk G = (M, \begin{psmatrix}0 & 1 \\ 1 & 0 \end{psmatrix}, \deg(\omega))$.  Let $\omega(H) = y^3x - \zeta^3 xy^3 +\zeta^2 yxy^2 -\zeta y^2xy +x^3y + \zeta^3 yx^3 + \zeta^2 xyx^2 + \zeta  x^2yx$,
where $\zeta$ is a primitive $8$th root of $1$.}\label{table.2vert1}
\begin{tabular}{cccccc}\hline
$\sigma$ & $\omega$ &  $M$ &  \\ \hline
$\left(\begin{smallmatrix} 0 & 1 & 0 \\ 1 & 0 & 0 \\ 0 & 0 & -1 \end{smallmatrix}\right)$ & $(xyz + zxy + yzx)  +( xzy + yxz + zyx)$ & $\left(\begin{smallmatrix} 1 & 2 \\ 2 & 1 \end{smallmatrix}\right)$ \\
$\operatorname{diag}(1, 1, -1)$ & $(xyz +  zxy + yzx) +(- xzy - yxz - zyx)$ & $\left(\begin{smallmatrix}2 & 1 \\ 1 & 2 \end{smallmatrix}\right)$ \\
$\operatorname{diag}(-1, -1, -1)$ & $(xyz +  zxy + yzx)+ ( - xzy - yxz - zyx)$ & $\left(\begin{smallmatrix} 0 & 3 \\ 3 & 0 \end{smallmatrix}\right)$ \\
$\operatorname{diag}(-1, 1)$ &  $\omega(H)$ & $\left(\begin{smallmatrix}1 & 1 \\ 1 & 1 \end{smallmatrix}\right)$ &
\end{tabular}
\end{table}

\begin{lemma}
\label{lem.2vert2}
Let $A \iso \kk Q/(\partial_a \omega : a \in Q_1, |Q_0|=2)$ satisfy Hypothesis~\ref{hyp.scy} with
type $\left(M, P = I, s=\deg(\omega)\right)$.
Then $M$ is one of the following:
\[
\deg(\omega)=3:
	\begin{pmatrix}1 & 2 \\ 2 & 1\end{pmatrix},
	\begin{pmatrix}2 & 1 \\ 1 & 2\end{pmatrix}, \qquad
\deg(\omega)=4:
	\begin{pmatrix}1 & 1 \\ 1 & 1\end{pmatrix},
    \begin{pmatrix}0 & 2 \\ 2 & 0\end{pmatrix}.
\]
Again, every such $M$ occurs for an algebra $A$ of the form $R \# \kk \mb{Z}_2$, where $R$ is an Artin-Schelter regular
algebra of dimension $3$.
\end{lemma}

\begin{proof}
Write $M=\begin{pmatrix}a&b\\c&d\end{pmatrix}$.
We have
\[ p(t) = I - Mt + M^T t^{s-1} - I t^s = \begin{pmatrix} 1 - at + at^{s-1} - t^s & -bt + ct^{s-1} \\ -ct + bt^{s-1} & 1 - dt + dt^{s-1} - t^s
\end{pmatrix}.\]
From this it is easy to calculate that $\left.\frac{d}{dt}\det(p(t))\right|_{t=1}=s (b-c)^2$.
By Theorem~\ref{thm.hilb2}(2), $\det(p(t))$ vanishes at $t = 1$ to order at least $3$.  Thus
$b=c$ and so $M$ is symmetric.
The eigenvalues of $M$ are now
\[ e_{\pm}=\frac{1}{2}\left((a+d)\pm \sqrt{(a-d)^2+4b^2}\right).\]
It is clear that $e_+ \geq |e_-|$ and so $e_+=\rho(M)$ is the spectral radius.

Suppose that $s=3$.   Since $M$ is symmetric and hence normal we know that it has spectral radius $\rho(M) = 3$ by
Theorem~\ref{thm.hilb2}(3), and in fact all of the eigenvalues of $M$ lie in the interval $[-1,3]$, by Theorem~\ref{thm.hilb2}(4).
Then $6 = (a+d) + \sqrt{(a-d)^2 + 4b^2}$, and rearranging
we get $(6-a-d)^2 = (a-d)^2 + (2b)^2$.  Thus $(a-d, 2b, 6-a-d)$ is a pythagorean triple, with
$0 \leq 6-a-d \leq 6$.  Since by hypothesis $M$ is the incidence matrix of a connected quiver $Q$, $b \neq 0$.  If $a-d \neq 0$, then we have a triple with three nonzero numbers, the largest of which is at most $6$, which gives only the solutions $(\pm 3, 4, 5)$; this forces $b = 2$, $a-d = \pm 3$, $6-a-d = 5$, which
has no solutions with $a, d \geq 0$.  Thus $a = d$,
$2b = 6-a-d = 6 - 2a$, and so we have  $(a, b) \in \{(0, 3), (1, 2), (2, 1) \}$.    However, $(a, b) = (0,3)$ is not valid,
since in this case the other eigenvalue $e_- = -3$ does not lie in the interval $[-1, 3]$ (one could also just check directly that $\det p(t)$ does not have all of its zeros on the unit circle in this case).

If $s= 4$, then $\rho(M) = 2$ by Theorem~\ref{thm.hilb2}.  Then a similar but easier analysis of pythagorean triples of the form $(a-d, 2b, 4-a-d)$ shows that $a = d$, and the corresponding possible solutions
are $(a, b) \in \{(0, 2), (1, 1) \}$.

The skew group algebras $R \# \kk \mb{Z}_2$ that achieve these types can be found in Table~\ref{table.2vert2}.
Note that a potential $\omega$ of the form $\omega = x^2 y^2 + a xy^2 x + a^2 y^2 x^2 + a^2 b yx^2 y$
with $a^2 b^2 = 1$ is of type $S_2$ in the Artin-Schelter classification \cite[Table 3.6]{AS}.
The corresponding derivation-quotient algebra
\[
A = \kk \langle x, y \rangle/( yx^2 + b x^2 y,  xy^2 + a y^2 x)
\]
is in fact AS regular for all such choices of $a$ and $b$, as can be checked, for example, by a Gr{\"o}bner basis calculation.
 This is relevant also for Table~\ref{table.3ver2} below.
\end{proof}

\begin{table}

\caption[2vert2]{$G = \ZZ_2 = \langle \sigma \rangle$ acting on $R=T(V)/(\partial_a \omega)$,
type of $R \# \kk G = (M, \begin{psmatrix}1 & 0 \\ 0 & 1 \end{psmatrix}, \deg(\omega))$}\label{table.2vert2}
\begin{tabular}{cccccc}\hline
$\sigma$ & $\omega$ & $M$ \\ \hline
$\operatorname{diag}(1, -1, -1)$ & $(xyz +  zxy + yzx) +(- xzy - yxz - zyx)$ & $\left(\begin{smallmatrix}1 & 2 \\ 2 & 1 \end{smallmatrix}\right)$ \\
$\left(\begin{smallmatrix} 0 & 1 & 0 \\ 1 & 0 & 0 \\ 0 & 0 & 1 \end{smallmatrix}\right)$ & $(xyz + zxy + yzx)  +(xzy + yxz + zyx)$ & $\left(\begin{smallmatrix} 2 & 1 \\ 1 & 2 \end{smallmatrix}\right)$ \\
$\operatorname{diag}(1, -1)$ & $y^2x^2 +xy^2x + x^2y^2 + yx^2y$ & $\left(\begin{smallmatrix}1 & 1 \\ 1 & 1 \end{smallmatrix}\right)$  \\
$\operatorname{diag}(-1, -1)$ & $y^2x^2 + xy^2x + x^2y^2 + yx^2y$ & $\left(\begin{smallmatrix}0 & 2 \\ 2 & 0 \end{smallmatrix}\right)$
 \end{tabular}
\end{table}

\section{The Three-vertex twisted case}
\label{sec.3vert1}

In this section we completely classify three-vertex quivers
supporting algebras satisfying Hypothesis \ref{hyp.scy}, where $P \neq I$, in other words where the Nakayama automorphism
gives a nontrivial permutation of the vertices of the quiver $Q$.  The incidence matrices $M$ appearing in this section will not necessarily
be symmetric, but they all turn out to be normal.   As in the two-vertex case, we split our analysis into cases depending on $P$.
Since we classify types only up to relabeling the vertices, it is enough to consider matrices $P$ corresponding to a particular $3$-cycle
and a particular $2$-cycle.

\begin{proposition}
\label{prop.3vert1}
Let $A \iso \kk Q/(\partial_a \omega : a \in Q_1, |Q_0|=3)$ satisfy Hypothesis \ref{hyp.scy} with type
$\left(M, P = \begin{psmatrix}0 & 1 & 0 \\0 & 0 &1 \\1 &0 & 0\end{psmatrix}, s=\deg(\omega)\right)$.
Then $M$ is one of the following:
\begin{align*}
&\deg(\omega)=3:
\begin{psmatrix}
	2 & 1 & 0 \\
	0&  2 & 1 \\
	1 & 0 & 2
\end{psmatrix},
\begin{psmatrix}
	1 & 0 & 2 \\
	2&  1 & 0 \\
	0 & 2 & 1
\end{psmatrix},
\begin{psmatrix}
	0 & 2 & 1 \\
	1&  0 & 2 \\
	2 & 1 & 0
\end{psmatrix},
\begin{psmatrix}
	1 & 1 & 1 \\
	1 & 1 & 1 \\
	1 & 1 & 1
\end{psmatrix}, \\
&\deg(\omega)=4:
\begin{psmatrix}
	0 & 1 & 1 \\
	1 & 0 & 1 \\
	1 & 1 & 0
\end{psmatrix},
\begin{psmatrix}
	1 & 1 & 0 \\
	0 & 1 & 1 \\
	1 & 0 & 1
\end{psmatrix},
\begin{psmatrix}
	1 & 0 & 1 \\
	1 & 1 & 0 \\
	0 & 1 & 1
\end{psmatrix},
\begin{psmatrix}
	0 & 2 & 0 \\
	0 & 0 & 2 \\
	2 & 0 & 0
\end{psmatrix}.
\end{align*}
Every listed type occurs for an algebra $A$ of the form $R \# \kk \mb{Z}_3$, where $R$ is an Artin-Schelter regular
algebra of dimension $3$.
\end{proposition}
\begin{proof}
We want to know the possible $M$ for which the matrix polynomial $p(t) = I - Mt + P M^Tt^{s-1} - P t^s$
has all roots on the unit circle.
The condition $PM = M P$ implies that $M$ is a circulant matrix by Proposition~\ref{prop.perms}, in other words
$M= \begin{psmatrix}a & b & c \\ c & a & b \\ b & c & a\end{psmatrix}$ for some $a, b, c \geq 0$.  In particular,
$M$ is normal, so the spectral radius of $M$ is $\rho(M) = 6-s$,  by Theorem~\ref{thm.hilb2}(3).  The spectral radius of $M$ is bounded below by the minimal row sum and above by the maximal row sum \cite[Theorem 8.1.22]{HJ}. Since all rows of $M$ have sum $a + b+ c$, the spectral radius of $M$ is $\rho(M) = a + b+ c$.

Suppose that $s = 3$, so $a+b+c = 3$.  The matrices $M$ and $P$ are commuting normal matrices and so are simultaneously diagonalizable.  Let $U = \begin{psmatrix} 1 & 1 & 1 \\ 1 & \zeta & \zeta^2 \\ 1 & \zeta^2 & \zeta \end{psmatrix}$, where $\zeta = e^{2\pi i/3}$.  Then $U^{-1}MU = D$, $U^{-1} P U  = Z$ where $D = \operatorname{diag}(3, \delta, \overline{\delta})$ and $Z = \operatorname{diag}(1, \zeta, \zeta^2)$, with $\delta =  a + b \zeta + c \zeta^2$.
We also have $U^{-1}M^T U = \operatorname{diag}(3, \overline{\delta}, \delta)$ since transposing $M$ switches the roles of $b$ and $c$.
Thus
\[
U^{-1} p(t) U = \operatorname{diag}(1 - 3t + 3t^2 - t^3, 1 - \delta t + \zeta \overline{\delta} t^2 - \zeta t^3, 1 - \overline{\delta} t + \overline{\zeta} \delta t^2 - \overline{\zeta}t^3).
\]

By Theorem~\ref{thm.hilb2}(2), $\det p(t) = \det U^{-1} p(t) U$ must have only roots of unity for zeros.
Thus the three polynomials in the diagonal of $U^{-1} p(t) U$ must have only roots of unity for zeros.
This is automatic for $1 - 3t + 3t^2 - t^3 = (1-t)^3$.  The other two polynomials are conjugate and so it is necessary and sufficient that $r_{a,b,c}(t) = 1 - \delta t + \zeta \overline{\delta} t^2 - \zeta t^3$ have only roots of unity for its zeros.  The
change of variable $t = \zeta u$ leads to $1 - \delta \zeta u + \overline{\delta} u^2 - \zeta u^3 = r_{c, a, b}(u)$,
and this has roots of unity for zeros if and only if $r_{a, b, c}(t)$ does.  Thus this condition is invariant under
cyclic permutation of $a, b, c$.

Now since $a+ b+ c = 3$, the set of possibilities is
\begin{align*}
(a, b, c) \in \{(3, 0, 0), (0, 3, 0), (0, 0 , 3), (1, 1, 1), (2, 1, 0), \\ (1, 2, 0), (2, 0, 1), (1, 0, 2), (0, 1, 2), (0, 2, 1) \}.
\end{align*}

When $(a, b, c) = (3, 0, 0)$ then $r_{a, b, c}(t) = 1 - 3t + 3\zeta t^2 - \zeta t^3$, which does not
have all of its roots on the unit circle, as is easy to check by computer.  Thus $(0, 3, 0)$ and $(0,0, 3)$ are also bad.

When $(a, b,c) = (2, 0, 1)$ then $r_{a,b,c}(t) = 1 - (2 + \zeta^2)t  + (2\zeta + \zeta^2)t^2 - \zeta t^3$,
which again does not have all of its roots on the unit circle.
Thus $(0, 1, 2)$ and $(1, 2, 0)$ are also bad.
The remaining matrices $M$ are the ones listed in the statement.

A similar analysis applies to $s = 4$, for which we must have $\rho(M) = a+b+c = 2$.  The list of possibilities is
\[ (a, b, c) \in \{(2, 0, 0), (0, 2, 0), (0, 0 , 2), (1, 1, 0), (1,0, 1), (0, 1, 1) \}.\]
In this case, with the same $U$ and $\delta$ as above we have
\[
U^{-1} p(t) U = \operatorname{diag}(1 - 2t +2t^3 - t^4, 1 - \delta t + \zeta \overline{\delta} t^3 - \zeta t^4, 1 - \overline{\delta} t + \overline{\zeta} \delta t^3 - \overline{\zeta}t^4).
\]
Again, $1 - 2t +2t^3 - t^4 = (1-t)^3(1+t)$ has all roots of unity for zeros, and the other two polynomials are
$s_{a, b, c}(t) = 1 - \delta t + \zeta \overline{\delta} t^3 - \zeta t^4$ and its conjugate, so we just need to investigate
 $s_{a,b,c}(t)$.  In this case, setting $t = \zeta u$, then $s_{a,b,c}(t) = \overline{s_{c, b, a}(u)}$, so $(a, b, c)$ is valid if and only if $(c, b, a)$ is.

When $(a, b, c) = (2, 0, 0)$, then one may check that
$s_{2,0,0} = 1 -2t + 2\zeta t^3 - \zeta t^4$ does not have roots of unity
for roots.  Thus $(0, 0, 2)$ is also bad.
Again, the remaining four possibilities are the ones listed.

We have narrowed the list down to four possible incidence matrices for $s = 3$ and four also for $s = 4$.   The
examples in Table~\ref{table.3vert1} show that they all occur as skew group algebras of $\mb{Z}_3$-actions on Artin-Schelter regular algebras, as claimed.   Note that the vertices of the McKay quiver are labeled by the irreducible characters $\chi_1, \chi_2, \chi_3$ as given
in the appendix, in that order.  The superpotential $\omega(E)$ is of type $E$ in the Artin Schelter
classification in \cite[Table 3.11]{AS}, and the corresponding derivation-quotient algebra is regular by \cite[Theorem 3.10]{AS}, since this potential is already generic in its family.  Similarly, $\omega'(E)$ is of type $E$
in \cite[Table 3.9]{AS} and the corresponding algebra is regular for the same reason.
\end{proof}

\begin{table}
\caption[3vert1]{$G = \ZZ_3 = \langle \sigma \rangle$ acting on acting on $R=T(V)/(\partial_a \omega)$, type of $R \# \kk G = \left(M, \begin{psmatrix}0 & 1 & 0 \\ 0 & 0 &1 \\ 1 &0 & 0 \end{psmatrix}, \deg(\omega)\right)$.  Let $\zeta$ be a primitive 3rd root of unity.   Let $\theta$ be a primitive 9th root of unity and let $\omega(E) =$ \\
$xzx + \theta^8 x^2z + \theta zx^2   + yxy + \theta^5 y^2 x + \theta^4 xy^2
+ zyz +  \theta^2 z^2 y + \theta^7 yz^2$.
\\
Let $\omega'(E)= y^3x + \zeta y^2xy + \zeta^2 yxy^2 + xy^3 + x^4$.
}\label{table.3vert1}
\begin{tabular}{cccc}\hline
$\sigma$ & $\omega$ & $M$ \\ \hline
$\operatorname{diag}(1, 1, \zeta^2)$ & $(xyz +  zxy + yzx) +(- xzy - yxz - zyx)$ & $\left(\begin{smallmatrix}2  & 1 & 0 \\ 0 & 2 & 1 \\ 1 & 0 & 2 \end{smallmatrix}\right)$  \\
 $\operatorname{diag}(1, \zeta, \zeta)$ & $(xyz +  zxy + yzx) +(- xzy - yxz - zyx)$ & $\left(\begin{smallmatrix} 1  & 0 & 2 \\2 & 1 & 0 \\ 0 & 2 & 1 \end{smallmatrix}\right)$ \\
$\operatorname{diag}(\zeta, \zeta^2, \zeta^2)$ & $(xyz +  zxy + yzx) +(- xzy - yxz - zyx)$ & $\left(\begin{smallmatrix}0  & 2 & 1 \\ 1 & 0 & 2 \\ 2 & 1 & 0 \end{smallmatrix}\right)$ \\
$\operatorname{diag}(1, \zeta, \zeta^2)$ & $\omega(E)$ & $\left(\begin{smallmatrix}1  & 1 & 1 \\ 1 & 1 & 1 \\ 1 & 1 & 1 \end{smallmatrix}\right)$ \\
$\operatorname{diag}(\zeta^2, \zeta)$ & $\omega'(E)$  & $\left(\begin{smallmatrix} 0   & 1 & 1 \\ 1 & 0 & 1 \\ 1  &1 & 0 \end{smallmatrix}\right)$ \\
$\operatorname{diag}(\zeta^2, 1)$ & $\omega'(E)$  & $\left(\begin{smallmatrix}1   & 1 & 0 \\ 0 & 1 & 1 \\ 1 & 0 & 1 \end{smallmatrix}\right)$ \\
$\operatorname{diag}(1, \zeta)$ & $x^2y^2 + xy^2x + y^2 x^2 + yx^2y$ & $\left(\begin{smallmatrix}1   & 0 & 1 \\ 1 & 1 & 0 \\ 0 & 1 & 1 \end{smallmatrix}\right)$ \\
$\operatorname{diag}(\zeta^2, \zeta^2)$ & $x^2y^2 + xy^2x + y^2 x^2 + yx^2y$ & $\left(\begin{smallmatrix}0   &2 & 0 \\ 0 & 0 & 2 \\ 2 & 0 & 0 \end{smallmatrix}\right)$
\end{tabular}
\end{table}

\bigskip

\begin{proposition}
\label{prop.3vert2}
Let $A \iso \kk Q/(\partial_a \omega : a \in Q_1, |Q_0|=3)$ satisfy Hypothesis \ref{hyp.scy} with type
$\left(M, P = \begin{psmatrix} 1 & 0 & 0 \\ 0 &0 & 1 \\ 0 & 1 & 0 \end{psmatrix}, s=\deg(\omega)\right)$.
Then $M$ is one of the following:
\begin{align*}
&\deg(\omega)=3:
\begin{psmatrix}
	2 & 1 & 1 \\
	1 & 1 & 0 \\
	1 & 0 & 1
\end{psmatrix},
\begin{psmatrix}
	2 & 1 & 1 \\
	1 & 0 & 1 \\
	1 & 1 & 0
\end{psmatrix},
\begin{psmatrix}
	1 & 1 & 1 \\
	1 & 0 & 2 \\
	1 & 2 & 0
\end{psmatrix}
\begin{psmatrix}
	1 & 1 & 1 \\
	1 & 1 & 1 \\
	1 & 1 & 1
\end{psmatrix}, \\
&\deg(\omega)=4:
\begin{psmatrix}
	1 & 1 & 1 \\
	1 & 0 & 0 \\
	1 & 0 & 0
\end{psmatrix},
\begin{psmatrix}
	0 & 1 & 1 \\
	1 & 1 & 0 \\
	1 & 0 & 1
\end{psmatrix}^*,
\begin{psmatrix}
	0 & 1 & 1 \\
	1 & 0 & 1 \\
	1 & 1 & 0
\end{psmatrix}^*.
\end{align*}
All of the listed types occur, except possibly the last two starred types.
\end{proposition}

\begin{proof}
The commutation $MP=PM$ forces $M$ to have the form
\[ M_Q = \begin{psmatrix}a & b & b \\ b' & c & d \\b' & d & c\end{psmatrix}\]
for some $a, b, b', c, d$, by Proposition~\ref{prop.perms}. Then
\[
p(t) = I - Mt + PM^T t^{s-1} - P t^s = \begin{psmatrix} 1 - at + at^{s-1} - t^s & -bt + b' t^{s-1} & -bt + b' t^{s-1} \\ -b't + bt^{s-1} & 1 -ct  + dt^{s-1} & -dt + ct^{s-1} - t^s \\ -b't + b t^{s-1} & -dt + ct^{s-1} - t^s & 1 - ct + dt^{s-1} \end{psmatrix}.
\]
So
\begin{gather*}
\det p(t) = 2(-bt + b't^{s-1})(-b't + bt^{s-1})(-1 + (c-d)t + (c-d)t^{s-1} - t^s) \\
+ (1-at + at^{s-1} -t^s)((1-ct + dt^{s-1})^2 - (-dt + ct^{s-1} - t^s)^2).
\end{gather*}
By Theorem~\ref{thm.hilb2}(2),
since $\GKdim(A) = 3$ by hypothesis, $1$ is a root of $\det p(t)$ of multiplicity at least $3$.
We have $\left.\det(p(t))\right|_{t=1} = 4(b'-b)^2(1-c+d)$, so either $b' = b$ or $1 - c + d = 0$.  If $1 -c + d = 0$, then the second summand in the formula for $\det p(t)$ already vanishes to order at least $3$ at $t = 1$, so the first summand
\[2(-bt + b't^{s-1})(-b't + bt^{s-1})(-1 + (c-d)t + (c-d)t^{s-1} - t^s)\] must also.
Since $s=3$ or $s = 4$, we have
\[\left.\frac{d^2}{dt^2} (-1 + (c-d)t + (c-d)t^{s-1} - t^s) \right|_{t=1} = (s-1)(s-2)(c-d) - s(s-1) \neq 0,\]
so either $(-bt+b't^{s-1})$ or $(-b't + bt^{s-1})$ vanishes at $t = 1$.  Thus $b = b'$ in this case as well.

So in any case we have $b = b'$.  Since $M$ is the incidence matrix of a connected quiver $Q$, $b \neq 0$.  The matrix $M$ is now
symmetric with eigenvalues
\[
e_{\pm} = \frac{1}{2}\left[(a + c + d)
\pm \sqrt{(a-c-d)^2 + 8b^2}\right], \;\;\;
e_0 = c - d.
\]
In particular, $M$ is normal, and so by Theorem~\ref{thm.hilb2}(3) its spectral radius must be
$\rho(M) = 6 - s$.  Obviously $e_{+} \geq |e_{-}|$ so either $e_{+}$ or $e_0$ is the spectral radius.

Suppose first that $s=3$, so the spectral radius of $M$ is $3$.  If $e_{0}$ is the spectral radius, then $c-d = 3$.  Since
$d \geq 0$, we have $c \geq 3$.  If also $a > 0$ or $d > 0$, then since $b > 0$ we get
$e_{+} > \frac{1}{2}( 4 + 2\sqrt{2}) > 3$, contradicting that $e_0$ is the spectral radius.  So $d = a  = 0$.  Then $c = 3$ and $(a-c-d) = -3$,
so that
$e_{+} > \frac{1}{2}(3 + \sqrt{17}) > 3$, again a contradiction.
Thus $e_{+} = 3$ is the spectral radius.  We see now that $0 \leq 6 - a - c - d \leq 6$ and
that $(6 - a -c - d)^2 = (a-c-d)^2 + 8b^2$, so that $(a-c-d, 2b, 6-a-c-d)$ is an integer solution to
$x^2 + 2y^2 = z^2$ with $0 \leq z \leq 6$ and $y > 0$.  The only such triples $(x, y, z)$ are $(\pm 1, 2, 3)$, and
$(\pm 2, 4, 6)$.  The triples $(\pm 2, 4, 6)$ force $a + c + d = 0$ and so $a =c = d = 0$, but then $a-c-d \neq \pm 2$, a contradiction.  If we have the solution $(1, 2, 3)$, then $a-c-d = 1, b = 1, 6-a-c-d = 3$, which gives
$a= 2, b= 1, c+d = 1$.  This leads to the first two matrices on the list.
The solution $(-1, 2, 3)$ gives $a=1, b=1, c+d = 2$, which leads
to three matrices.   When $a = 1, b = 1, c= 2, d = 0$, a simple computer check shows that $p(t)$ does not have all of
its zeros on the unit circle.  (This can also be easily shown by diagonalizing the matrix polynomial, similarly as in the preceding result.)
The other two cases $c = d = 1$ and $c = 0, d = 2$ give the rest of the matrices on the list for $s = 3$.

Now suppose that $s = 4$, so the spectral radius is $2$.  If $e_0$ is the spectral radius, then $c-d = 2$, so
$c \geq 2$ and $b > 0$ implies that $e_{+} \geq \frac{1}{2}(2 + 2 \sqrt{2}) > 2$, a contradiction.
So $e_{+} = 2$ is the spectral radius, and $(a-c-d, 2b, 4-a-c-d)$ is an integer solution to
$x^2 + 2y^2 = z^2$ with $0 \leq z \leq 4$ and $y > 0$.  Now we only have the solutions
$(\pm 1, 2, 3)$ to worry about.  The solution $(1, 2, 3)$ leads to the possibility $a = 1, b=1, c=d= 0$. The solution $(-1, 2, 3)$ leads to $a =0, b = 1, c+d = 1$, which gives two more matrices with $c=1, d= 0$ or $c = 0, d = 1$.  These are the three matrices on the list for $s = 4$.

This time only some of the types---the first two for $s = 3$ and the first one for $s = 4$---can be realized directly as the types of algebras Morita equivalent to skew group algebras of the form $R \# \kk G$ for Artin-Schelter regular algebras $R$.  When this works, $G = S_3$ is the symmetric group.  See Table~\ref{table.3ver2}.

We find examples of the other types with $s = 3$ using Ore extensions, as follows.
Recall that $P = \begin{psmatrix} 1 & 0 & 0 \\ 0 &0 & 1 \\ 0 & 1 & 0 \end{psmatrix}$ is the permutation matrix of the action
of the Nakayama automorphism on the vertices that we are looking for.  Let $\zeta$ be a primitive third root of unity.

In Example~\ref{ex.component}, we gave a group $G$ acting on a 2-dimensional vector space
such that the McKay quiver breaks up into connected components,
one of which has the incidence matrix $M' = \begin{psmatrix} 0 & 1 & 1 \\ 1 & 1 & 0 \\ 1 & 0 & 1 \end{psmatrix}$.
As seen in that example, $\kk[x, y] \# \kk G$ is twisted Calabi-Yau and breaks up as a product of algebras,
one of which is Morita equivalent to a twisted Calabi-Yau algebra $B$ of dimension $2$, where $B \cong \kk Q/I$
with $Q$ of incidence matrix $M'$, and where the Nakayama automorphism of $B$ acts trivially on the vertices of $Q$.
We also noted there that $B$ has an automorphism $\sigma$ that permutes the two vertices with loops,
so $\sigma$ acts on vertices via the matrix $P$.
By the discussion of Ore extensions in Section~\ref{sec.motive},
the algebra $A = B[t; \sigma]$ is twisted Calabi-Yau of dimension $3$ of type $(M = M' + P, P = PI, 3)$, where here
$M = \begin{psmatrix} 1 & 1 & 1 \\ 1 & 1 & 1 \\ 1 & 1 & 1 \end{psmatrix}$.

Similarly, there is a twisted Calabi-Yau algebra $B = \kk Q/I$ of dimension $2$, where $Q$ has incidence matrix
$M' = \begin{psmatrix} 0 & 1 & 1 \\ 1 & 0 & 1 \\ 1 & 1 & 0 \end{psmatrix}$,
namely $B = \kk[x, y] \# \kk \mb{Z}_3$ where $\mb{Z}_3 = \langle \sigma \rangle$ acts via $\sigma(x) = \zeta x$, $\sigma(y) = \zeta^2 y$
for a primitive third root of unity $\zeta$.  Again this action has trivial $\hdet$ and so the Nakayama automorphism of $B$ acts trivially on the vertices of $Q$.  Now $B$ has an automorphism $\sigma$ that acts on the vertices via $P$, which is induced from the automorphism
of $k \mb{Z}_3$ that interchanges $\sigma$ and $\sigma^2$.
So $A = B[t; \sigma]$ is twisted Calabi-Yau of dimension $3$ of type $(M = M' + P, P = PI, 3)$ where
$M = \begin{psmatrix} 1 & 1 & 1 \\ 1 & 0 & 2 \\ 1 & 2 & 0 \end{psmatrix}$.

Now all the listed types except the final two starred ones have been attained.
\end{proof}

\begin{table}
\caption[3ver2]{$G = S_3 = \langle \sigma,\tau \rangle$ acting on acting on $R=T(V)/(\partial_a \omega)$, where $\sigma = (123)$ and $\tau = (12)$, type of $R \# \kk G = \left(M,\begin{psmatrix}1& 0 & 0 \\ 0 & 0 & 1 \\ 0 &  1 & 0 \end{psmatrix},\deg(\omega)\right)$. }\label{table.3ver2}
\begin{tabular}{cccc}\hline
$\sigma$ & $\tau$ & $\omega$ & $M$ \\ \hline
$\left(\begin{smallmatrix}0  & 0 & 1 \\ 1 & 0 & 0 \\ 0 & 1 & 0 \end{smallmatrix}\right)$ &
$\left(\begin{smallmatrix}0 & 1 &  0 \\1 & 0 & 0 \\ 0 & 0 & 1 \end{smallmatrix}\right)$ &
$(xyz +  zxy + yzx) +(- xzy - yxz - zyx)$ &
$\left(\begin{smallmatrix}2 &1 &  1 \\1 & 1 & 0 \\ 1 & 0 & 1 \end{smallmatrix}\right)$ \\
$\left(\begin{smallmatrix}0  & 0 & 1 \\ 1 & 0 & 0 \\ 0 & 1 & 0 \end{smallmatrix}\right)$ &
$\left(\begin{smallmatrix}0 & -1 &  0 \\-1 & 0 & 0 \\ 0 & 0 & -1 \end{smallmatrix}\right)$ &
$(xyz +  zxy + yzx) +(xzy+ yxz + zyx)$ &
$\left(\begin{smallmatrix}2 & 1  &  1 \\1 & 0 & 1 \\ 1 & 1 & 0 \end{smallmatrix}\right)$ \\
$\operatorname{diag}(\zeta, \zeta^2)$ &
$\left(\begin{smallmatrix}0 & 1 \\1 & 0 \end{smallmatrix}\right)$ & $x^2y^2 + i xy^2x - y^2x^2 -i yx^2y$ &
$\left(\begin{smallmatrix}1 &1  &  1 \\1 & 0 & 0 \\ 1 & 0 & 0 \end{smallmatrix}\right)$
\end{tabular}
\end{table}

\noindent
It seems likely that similar methods as those used in this section can be used to classify the types of twisted Calabi-Yau algebras
on more than three vertices, when the matrix $P$ is a large cycle or a product of large cycles, but we do not attempt to analyze more than three vertices in this paper.

\section{The three-vertex case with $P = I$}

We are now ready to attack the case of twisted Calabi-Yau algebras satisfying Hypothesis~\ref{hyp.scy} for which the quiver has three vertices and the type has $P = I$.  As already mentioned, this is the most difficult case, and we rely on computer programs to help at several key points.

The first step is to identify matrices $M \in M_3(\mb{Z}_{\geq 0})$ that satisfy the restriction
given by Theorem~\ref{thm.hilb2} that $\det p(t)$ has only roots of unity for zeroes, where
$p(t) = I - Mt + M^Tt^{s-1} - It^s$ for $s = 3$ or $s = 4$.  One can also check that $\det p(t)$ has a root
of at least multiplicity $3$ at $t = 1$, as also proscribed by Theorem~\ref{thm.hilb2}.   This is a difficult but doable task, once
we also implement the restrictions on the number of loops from Proposition~\ref{prop.loops}.  This will leave us with a long list, including some infinite families.
In the previous section where we assumed $P \neq I$, all matrices $M$ we found satisfying the already discussed restrictions on $p(t)$ turned out to occur as part of a valid type (except possibly the starred matrices in Proposition~\ref{prop.3vert2} for which we are unsure).
This will not be the case when $P = I$.  Hence, our second task will be to eliminate in some way those
matrices $M$ that cannot be part of a type of a twisted Calabi-Yau algebra for some other reason.
Some of them are normal, but do not have the correct
spectral radius forced by Theorem~\ref{thm.hilb2} in the normal case.  Some others lead to a putative matrix Hilbert series $p(t)^{-1}$
in which not all coefficients are nonnegative, so cannot be the actual Hilbert series of an algebra.   We prove a general
result (Lemma~\ref{lem.partition}) which shows that there cannot be a good superpotential of degree $s$ when the arrows
in the quiver can be partitioned into two nonempty sets such at every $s$-cycle in the quiver involves only arrows in one or the other set.
This lemma eliminates several more families.
Finally, we eliminate a few others by some more ad-hoc Gr\"{o}bner basis methods, leaving a short list of matrices $M$ that we can actually show occur as part of a type.

\subsection{Finding the long list}

We turn now to the first step.   Let $M \in M_3(\mb{Z}_{\geq 0})$ where $(M, I, s)$ is the type of a twisted Calabi-Yau algebra satisfying
Hypothesis~\ref{hyp.scy}.  Consider $\det p(t)$ for $p(t) = (I - Mt + M^Tt^{s-1} - It^s)$, where $s = 3$ or $s = 4$.  This is a polynomial of degree $3s$ in $\mb{Z}[t]$, and all of its zeros must be roots of unity, with $1$ as a root of multiplicity at least 3, by Theorem~\ref{thm.hilb2}(2).  In particular, $\det p(t)$ must be a product of cyclotomic polynomials.  The matrix polynomial $p(t)$ also satisfies the equation $-t^s p(t^{-1})^T = p(t)$ by direct calculation, and so $\det p(t) = -t^{3s} \det p(t^{-1})$, which says that $\det p(t)$ is antipalindromic.
When $s = 4$, then $\det p(t)$ is antipalindromic of
even degree and thus has $t = -1$ as a root.  Combining these facts we have
\begin{equation}
\label{eq.pform}
\det(p(t))=\begin{cases}
(1-t)^3r(t) & \text{ if } s=3 \\
(1-t)^3(1+t)r(t) & \text{ if } s=4,\end{cases}
\end{equation}
where $r(t)$ is a product of cyclotomic polynomials with $\deg(r(t))=2s$.  Moreover, the cyclotomic polynomials $\Phi_n(t)$ are palindromic when $n > 1$, while $\Phi_1(t) = t-1$ is antipalindromic.  Thus $1$ occurs as a root of $r(t)$ with even multiplicity.   The roots of $r(t)$ other than $1$ and $-1$ must appear in inverse pairs $(\alpha, 1/\alpha)$.  Thus $-1$ also occurs as a root of $r(t)$ with even multiplicity.
Now we see that
\begin{align}
\label{eq.rform}
r(t) = \prod_{i=1}^d (1-k_i t + t^2), \;\;
k_i \in \RR, \;\; |k_i| \leq 2 \text{ for each } i.
\end{align}
This follows since for a root $\alpha \neq \pm 1$ of $r(t)$ we have $(t-\alpha)(t - \overline{\alpha}) = t^2 - 2\operatorname{Re}(\alpha) + 1$
where $|\operatorname{Re}(\alpha)| \leq 1$ because $\alpha$ is a root of unity; while $1$ and $-1$ occur with even multiplicity and $(t-1)^2 = (t^2 - 2t + 1)$
and $(t+1)^2 = t^2 + 2t + 1$ also have the correct form.

Now suppose that $M$ has diagonal entries $(u,v,w)$.  Proposition~\ref{prop.loops} limits the number of loops of a vertex of the quiver with incidence
matrix $M$ to $6-s$ in the cases $s =3, 4$ at hand.  Thus we can assume that $0 \leq u,v,w \leq 6-s$.  Since we are classifying types only up to permutation
of the vertices, we can assume that $u \geq v \geq w$.   Moreover, if $A$ has type $(M, I, s)$ then its opposite algebra $A^{\op}$ is also twisted Calabi-Yau (see the remark after \cite[Definition 4.1]{RR2}).  Since $(\kk Q)^{\op} \cong \kk Q^{\op}$, where $Q^{\op}$ is the opposite quiver formed by reversing the direction of all of the arrows, we see that $A^{\op}$ has type $(M^T, I, s)$.  Thus we can also classify the matrices $M$ up to transpose.

For each of the finitely many possibilities for $(u,v,w)$, we can use Maple to look for
possible products of cyclotomic polynomials whose $t^{3s-1}$ coefficient is $u+v+w$.
However, this leaves six unknown parameters in $M_Q$, which is still computationally difficult to solve for.
Therefore, we seek a way to limit these parameters.

Write $M=[m_{ij}]_{i,j=1...3}$ and set
\[
\lambda=\sum_{i=1}^3 m_{ii} = u + v + w, \;\;
\beta=\sum_{1 \leq i < j \leq 3} m_{ii}m_{jj}, \;\;
\gamma = \sum_{1 \leq i < j \leq 3} m_{ij}m_{ji}.\]
The coefficient of $t^{3s-1}$ in $\det(p(t))$ is $\lambda$.
If $s=3$, then the coefficient of $t^{3s-2}$ in $\det(p(t))$ is $\gamma-\beta-\lambda$
and if $s=4$ then it is $\gamma-\beta$.

Comparing to \eqref{eq.pform} and \eqref{eq.rform}
gives the following equations:

\begin{align*}
&\sum_{i=1}^s k_i = \lambda-(6-s),  \\
&\gamma = \beta - 2\lambda - 3(s-4) - \sum_{1 \leq i < j \leq s} k_ik_j.
\end{align*}

Now the idea is that for a fixed choice of $(u,v,w)$, we know $\lambda$ and $\beta$, and the second
equation gives $\gamma$ in terms of the $k_i$.  Since $|k_i| \leq 2$ for all $i$ we can maximize the value of $\gamma$
on a compact set, given the constraint from the first equation.  This is a simple Lagrange multipliers problem,
and used Mathematica to determine the maximum value of gamma
corresponding to each set of diagonal entries.
This maximum value of $\gamma$ is recorded in Table~\ref{table.max1} and Table~\ref{table.max2}.
\begin{table}
\caption{$\deg(\omega)=3$}\label{table.max1}
\begin{tabular}{|c|c|c|c|c|c|c|c|c|c|c|c|c|c|}
\hline
diagonal entries & $\lambda$ & $\beta$ & $\gamma_{max}$ & &
diagonal entries & $\lambda$ & $\beta$ & $\gamma_{max}$ \\
\hline
(3,3,3) 	& 9 & 27	&	0	& & 		(2,1,1) 	& 4 & 5		&	4 \\	
(3,3,2)	& 8 & 21 	&	0 	& & 		(2,2,0)	& 4 & 4		&	3 \\	
(3,2,2)	& 7 & 16	& 	1	& & 		(3,1,0)	& 4 & 3		&	2 \\		
(3,3,1)	& 7 & 15	& 	0	& & 		(1,1,1)	& 3 & 3		&	4	\\
(2,2,2)	& 6 & 12	& 	3	& & 		(2,1,0)	& 3 & 2		&	3 \\	
(3,2,1)	& 6 & 11	& 	2	& & 		(3,0,0) 	& 3 & 0		&	1	\\
(3,3,0)	& 6 & 9	& 	0	& & 		(1,1,0) 	& 2 & 1		&	4	\\		 	
(2,2,1)	& 5 & 8	&	5	& &		(2,0,0) 	& 2 & 0		&	3	\\	
(3,1,1)	& 5 & 7	&	4	& &		(1,0,0)	& 1 & 0		&	5	\\
(3,2,0) 	& 5 & 6	&	3	& &		(0,0,0) 	& 0 & 0		&	3	\\
\hline
\end{tabular}
\end{table}

\begin{table}
\caption{$\deg(\omega)=4$}\label{table.max2}
\begin{tabular}{|c|c|c|c|c|c|c|c|c|c|c|c|c|c|}
\hline
diagonal entries & $\lambda$ & $\beta$ & $\gamma_{max}$ & &
diagonal entries & $\lambda$ & $\beta$ & $\gamma_{max}$ \\
\hline
(2,2,2)	& 6 & 12 	& 	0 	& & 		(2,1,0)	& 3 & 2		&	2	\\
(2,2,1)	& 5 & 8		&	0	& &	 	(1,1,0) & 2 & 1		&	5	\\	
(2,1,1)	& 4 & 5		&	1	& &		(2,0,0) & 2 & 0		&	4	\\
(2,2,0)	& 4 & 4		&	0	& &		(1,0,0)	& 1 & 0		&	4   \\
(1,1,1) & 3 & 3		&	3	& &		(0,0,0) & 0 & 0		&	4	\\ 		
\hline
\end{tabular}
\end{table}

Knowing the maximum value of $\gamma$ limits the off-diagonal entries of $M_Q$.
For example, if $\gamma=2$ then $M_Q$ may only have the following forms
up to transposition and permutation of the vertices:
\begin{align*}
&
\begin{psmatrix}u & x & 0 \\ 0 & v & 2 \\ y & 1 & w\end{psmatrix}, \;\;
\begin{psmatrix}u & x & 2 \\ 0 & v & y \\ 1 & 0 & w\end{psmatrix}, \;\;
\begin{psmatrix}u & 2 & x \\ 1 & v & 0 \\ 0 & y & w\end{psmatrix}, \\
&
\begin{psmatrix}u & x & 0 \\ 0 & v & 1 \\ y & 2 & w\end{psmatrix}, \;\;
\begin{psmatrix}u & x & 1 \\ 0 & v & y \\ 2 & 0 & w\end{psmatrix}, \;\;
\begin{psmatrix}u & 1 & x \\ 2 & v & 0 \\ 0 & y & w\end{psmatrix}, \\
&
\begin{psmatrix}u & x & 1 \\ 0 & v & 1 \\ 1 & 1 & w\end{psmatrix}, \;\;
\begin{psmatrix}u & 1 & x \\ 1 & v & 1 \\ 0 & 1 & w\end{psmatrix}, \;\;
\begin{psmatrix}u & 1 & 1 \\ 1 & v & y \\ 1 & 0 & w\end{psmatrix}.
\end{align*}

Of course the number of forms grows as $\gamma$ does but it is finite
and computationally easy to manage.
We now run the following algorithm in Maple\footnote{Code available at code.ncalgebra.org.}.
\begin{itemize}
\item Fix $s = 3$ or $s = 4$ and choose a set of diagonal entries $(u,v,w)$ with $0 \leq u, v, w \leq 6-s$.
\item Choose $\gamma$ such that $\gamma \leq \gamma_{max}$ from Table \ref{table.max1} or \ref{table.max2}.
\item Choose a form $M$ corresponding to the $\gamma$ value.
\item Choose a product of cyclotomic polynomials $q(t)$ whose $t^{3d-1}$-coefficient is equal to $u+v+w$.
\item Compute the matrix polynomial $p(t)$ of $M$ and check whether there are
values of the unknown parameters such that $\det p(t) = q(t)$.
\end{itemize}

Here is the result of this calculation.
\begin{lemma}
\label{lem.goodpoly}
Let $A \iso \kk Q/(\partial_a \omega : a \in Q_1, |Q_0|=3)$ be a twisted Calabi-Yau derivation-quotient algebra satisfying Hypothesis~\ref{hyp.scy} with type $(M, I, s)$, where $s = \deg(\omega)$.  Up to permutation of the vertices and transposition, the incidence matrix $M$ of the quiver $Q$
must be on the following list.  Throughout, let $\alpha \in [0,3]$ and $x>0$ be integers.
Let $a,b,c$ be integers satisfying $a^2+b^2+c^2=abc$.

$s=3$:
\begin{align*}
&\begin{psmatrix}3 & x & 0 \\ 0 & 0 & 2 \\ x & 0 & 0\end{psmatrix},
\begin{psmatrix}3 & 0 & 1 \\ 1 & 2 & 0 \\ 1 & 1 & 1\end{psmatrix},
\begin{psmatrix}3 & 1 & 0 \\ 1 & 1 & 2 \\ 2 & 0 & 0\end{psmatrix},
\begin{psmatrix}3 & 1 & 0 \\ 1 & 2 & 1 \\ 2 & 0 & 0\end{psmatrix},
\begin{psmatrix}3 & 0 & 3 \\ 2 & 2 & 0 \\ 1 & 1 & 0\end{psmatrix},
\begin{psmatrix}3 & 0 & 1 \\ 0 & 2 & 2 \\ 1 & 1 & 0\end{psmatrix},
\begin{psmatrix}2 & 1 & 0 \\ 1 & 2 & 1 \\ 1 & 0 & 0\end{psmatrix},
\begin{psmatrix}2 & 0 & 1 \\ 0 & 2 & 1 \\ 1 & 1 & 0\end{psmatrix},
\begin{psmatrix}2 & 0 & 3 \\ 10 & 2 & 0 \\ 1 & 4 & 0\end{psmatrix}, \\
&\begin{psmatrix}2 & 0 & 1 \\ 0 & 1 & 1 \\ 1 & 1 & 0\end{psmatrix},
\begin{psmatrix}2 & 1 & 1 \\ 1 & 1 & 0 \\ 1 & 0 & 0\end{psmatrix},
\begin{psmatrix}2 & 1 & 0 \\ 1 & 1 & 1 \\ 0 & 1 & 0\end{psmatrix},
\begin{psmatrix}2 & 1 & 1 \\ 1 & 0 & 0 \\ 1 & 0 & 0\end{psmatrix},
\begin{psmatrix}2 & 3 & 0 \\ 1 & 0 & 2 \\ 4 & 0 & 0\end{psmatrix},
\begin{psmatrix}2 & 1 & 1 \\ 1 & 0 & 1 \\ 1 & 1 & 0\end{psmatrix},
\begin{psmatrix}0 & 1 & 1 \\ 1 & 0 & 1 \\ 1 & 1 & 0\end{psmatrix},
\begin{psmatrix}3 & 1 & 0 \\ 1 & 1 & 3 \\ 2 & 1 & 1\end{psmatrix},
\begin{psmatrix}\alpha & x & 0 \\ 0 & 1 & 2 \\ x & 2 & 1\end{psmatrix},\\
&\begin{psmatrix}3 & 2 & 1 \\ 0 & 1 & 3 \\ 1 & 1 & 1\end{psmatrix},
\begin{psmatrix}2 & 2 & 0 \\ 1 & 1 & 1 \\ 1 & 0 & 1\end{psmatrix},
\begin{psmatrix}2 & 0 & 1 \\ 0 & 1 & 1 \\ 1 & 1 & 1\end{psmatrix},
\begin{psmatrix}2 & 1 & 1 \\ 1 & 1 & 0 \\ 1 & 0 & 1\end{psmatrix},
\begin{psmatrix}1 & 0 & 1 \\ 0 & 1 & 1 \\ 1 & 1 & 1\end{psmatrix},
\begin{psmatrix}1 & 1 & 1 \\ 1 & 1 & 1 \\ 1 & 1 & 1\end{psmatrix},
\begin{psmatrix}1 & 0 & 1 \\ 0 & 1 & 1 \\ 1 & 1 & 0\end{psmatrix},
\begin{psmatrix}1 & 1 & 1 \\ 1 & 1 & 0 \\ 1 & 0 & 0\end{psmatrix},
\begin{psmatrix}1 & 1 & 1 \\ 1 & 1 & 1 \\ 1 & 1 & 0\end{psmatrix},\\
&\begin{psmatrix}1 & 0 & 1 \\ 1 & 2 & 1 \\ 0 & 1 & 2\end{psmatrix},
\begin{psmatrix}1 & 1 & 1 \\ 1 & 2 & 0 \\ 1 & 0 & 2\end{psmatrix},
\begin{psmatrix}1 & 1 & 1 \\ 1 & 0 & 0 \\ 1 & 0 & 0\end{psmatrix},
\begin{psmatrix}1 & 1 & 1 \\ 1 & 0 & 1 \\ 1 & 1 & 0\end{psmatrix},
\begin{psmatrix}1 & 4 & 0 \\ 1 & 0 & 2 \\ 5 & 0 & 0\end{psmatrix},
\begin{psmatrix}1 & 2 & 1 \\ 1 & 0 & 2 \\ 2 & 0 & 0\end{psmatrix},
\begin{psmatrix}0 & a & 0 \\ 0 & 0 & b \\ c & 0 & 0\end{psmatrix}.
\end{align*}

$s=4$:
\begin{align*}
\begin{psmatrix}1 & 1 & 0 \\ 0 & 1 & 1 \\ 1 & 0 & 1\end{psmatrix},
\begin{psmatrix}1 & 1 & 0 \\ 0 & 1 & 1 \\ 2 & 0 & 1\end{psmatrix},
\begin{psmatrix}1 & 0 & 1 \\ 0 & 0 & 1 \\ 1 & 1 & 0\end{psmatrix},
\begin{psmatrix}1 & 1 & 1 \\ 1 & 0 & 0 \\ 1 & 0 & 0\end{psmatrix},
\begin{psmatrix}0 & 0 & 1 \\ 0 & 0 & 1 \\ 1 & 1 & 0\end{psmatrix},
\begin{psmatrix}0 & 1 & 1 \\ 1 & 0 & 1 \\ 1 & 1 & 0\end{psmatrix},
\begin{psmatrix}1 & 0 & 1 \\ 0 & 1 & 1 \\ 1 & 1 & 0\end{psmatrix}.
\end{align*}
\end{lemma}

\subsection{Narrowing down the list}

Next, we comb through the matrices in Lemma \ref{lem.goodpoly}
and remove those which cannot actually occur as an incidence matrix in a type for various reasons.

\begin{lemma}
(1) The following matrices are normal but have spectral radius $\rho(M) \neq 6 -s$, so $(M, I, s)$ cannot
be the type of a twisted Calabi-Yau algebra of dimension $3$ by Theorem~\ref{thm.hilb2}.
\begin{align*}
s=3:
&\begin{psmatrix}2 & 0 & 1 \\ 0 & 2 & 1 \\ 1 & 1 & 0\end{psmatrix},
\begin{psmatrix}2 & 0 & 1 \\ 0 & 1 & 1 \\ 1 & 1 & 0\end{psmatrix},
\begin{psmatrix}2 & 1 & 1 \\ 1 & 1 & 0 \\ 1 & 0 & 0\end{psmatrix},
\begin{psmatrix}2 & 1 & 0 \\ 1 & 1 & 1 \\ 0 & 1 & 0\end{psmatrix},
\begin{psmatrix}2 & 1 & 1 \\ 1 & 0 & 0 \\ 1 & 0 & 0\end{psmatrix},
\begin{psmatrix}0 & 1 & 1 \\ 1 & 0 & 1 \\ 1 & 1 & 0\end{psmatrix},
\begin{psmatrix}2 & 0 & 1 \\ 0 & 1 & 1 \\ 1 & 1 & 1\end{psmatrix},\\
&\begin{psmatrix}1 & 0 & 1 \\ 0 & 1 & 1 \\ 1 & 1 & 1\end{psmatrix},
\begin{psmatrix}1 & 0 & 1 \\ 0 & 1 & 1 \\ 1 & 1 & 0\end{psmatrix},
\begin{psmatrix}1 & 1 & 1 \\ 1 & 1 & 0 \\ 1 & 0 & 0\end{psmatrix},
\begin{psmatrix}1 & 1 & 1 \\ 1 & 1 & 1 \\ 1 & 1 & 0\end{psmatrix},
\begin{psmatrix}1 & 1 & 1 \\ 1 & 0 & 0 \\ 1 & 0 & 0\end{psmatrix},
\begin{psmatrix}1 & 1 & 1 \\ 1 & 0 & 1 \\ 1 & 1 & 0\end{psmatrix}. \\
s=4:
&\begin{psmatrix}1 & 0 & 1 \\ 0 & 0 & 1 \\ 1 & 1 & 0\end{psmatrix},
\begin{psmatrix}0 & 0 & 1 \\ 0 & 0 & 1 \\ 1 & 1 & 0\end{psmatrix}.
\end{align*}

(2) Assume $s=3$.  For each of the following $M$, the matrix Hilbert series $(I - Mt + M^Tt^2 - It^3)^{-1}$
does not have nonnegative coefficients, so again  $(M, I, s)$ cannot
be the type of a twisted Calabi-Yau algebra of dimension $3$ by Theorem~\ref{thm.hilb2}.
\begin{align*}
&\begin{psmatrix}3 & 1 & 0 \\ 0 & 0 & 2 \\ 1 & 0 & 0\end{psmatrix},
\begin{psmatrix}3 & 0 & 1 \\ 1 & 2 & 0 \\ 1 & 1 & 1\end{psmatrix},
\begin{psmatrix}3 & 1 & 0 \\ 1 & 1 & 2 \\ 2 & 0 & 0\end{psmatrix},
\begin{psmatrix}3 & 1 & 0 \\ 1 & 2 & 1 \\ 2 & 0 & 0\end{psmatrix},
\begin{psmatrix}3 & 0 & 1 \\ 0 & 2 & 2 \\ 1 & 1 & 0\end{psmatrix},
\begin{psmatrix}2 & 1 & 0 \\ 1 & 2 & 1 \\ 1 & 0 & 0\end{psmatrix},
\begin{psmatrix}3 & 1 & 0 \\ 1 & 1 & 3 \\ 2 & 1 & 1\end{psmatrix},
\begin{psmatrix}3 & 2 & 1 \\ 0 & 1 & 3 \\ 1 & 1 & 1\end{psmatrix},
\begin{psmatrix}1 & 0 & 1 \\ 1 & 2 & 1 \\ 0 & 1 & 2\end{psmatrix}.
\end{align*}
\end{lemma}

Next, we show that quivers of a certain shape cannot possibly have a superpotential
whose associated derivation-quotient algebra satisfies Hypothesis \ref{hyp.scy},
even without requiring polynomial growth.
\begin{lemma}
\label{lem.partition}
Let $Q$ be a connected quiver with a $\sigma$-twisted superpotential of degree $s \geq 3$, where $\sigma$ acts trivially on the vertices of $Q$, and where $Q$ has incidence matrix $M$.  Let $A =  \kk Q/(\partial_a \omega : a \in Q_1)$ be the corresponding
derivation-quotient algebra.

Suppose that the set $Q_1$ of arrows in $Q$ can be partitioned into two nonempty disjoint
subsets $S_1$ and $S_2$, such that every $s$-cycle in $Q$ has all of its arrows in either $S_1$ or in $S_2$.
Then the matrix Hilbert series of $h_A(t)$ disagrees with the
Hilbert series $(I - Mt + Mt^{s-1} - I t^s)^{-1}$ in some degree $\leq s$
and so $A$ cannot be twisted Calabi-Yau of dimension $3$.
\end{lemma}
\begin{proof}
Suppose that the Hilbert series of $h_A(t)$ agrees with $(I - Mt + Mt^{s-1} - I t^s)^{-1}$ in degrees
up to $s$.  An easy calculation of the first
terms of this series up to degree $s$ gives
\begin{align*}
h_A(t) &= I + Mt + M^2t^2 + \dots + M^{s-2} t^{s-2} \\
	&\quad + (M^{s-1} - M^T) t^{s-1} + (M^s - M^TM - MM^T + I) t^s + \dots
\end{align*}
A relation $\partial_a \omega$ which goes from $i$ to $j$ must come from an arrow $a$ from $j$ to $i$.  Thus there
are $(M^T)_{ij}$ such relations, and they all have degree $s-1$.  Since the Hilbert series predicts that
$\dim_{\kk} e_i A_{s-1} e_j = M^{s-1}_{ij} - M^T_{ij}$, where $M^{s-1}_{ij} = \dim_{\kk} e_i \kk Q_{s-1} e_j$,
 as $a$ runs over the arrows from $j$ to $i$ the obtained relations $\partial_a \omega$ must be linearly independent.
Thus the $\kk$-span $R$ of the relations  $\partial_a \omega$ satisfies
$\dim_{\kk} e_i R_{s-1} e_j = (M^T)_{ij}$.

Let $V$ be the $\kk$-span of the arrows in
$\kk Q$.  To get the degree $s$ relations we multiply on either side by arrows; thus if $I = (\partial_a \omega : a \in Q_1)$ is the ideal of
relations then $I_s = VR + RV$.  From this we see that
\begin{align*}
\dim_{\kk} e_i I_s e_j
	&= \dim_{\kk} e_i (RV + VR) e_j \\
	&= \dim_{\kk} e_i RV e_j + \dim_{\kk} e_i VR e_j - \dim_{\kk} e_i (RV \cap VR)e_j \\
	&= (MM^T + M^T M)_{ij} - \dim_{\kk} e_i (RV \cap VR) e_j.
\end{align*}
Comparing with $h_A(t)$ above, observing that $\dim_{\kk} (M^s)_{ij}$ is the number of paths in $\kk Q$ of length $s$ from $i$ to $j$, we see that $\dim_{\kk} e_i (RV \cap VR) e_j = \delta_{ij}$.  In particular, for each $i$ there is exactly one relation up to scalar which is in both $RV$ and $VR$ and goes from vertex $i$ to itself.

Now since $\omega$ is a linear combination of $s$-cycles, by the hypothesis we may uniquely write $\omega= \omega_1 + \omega_2$, where $\omega_1$ is a linear combination of cycles involving only arrows in $S_1$ and $\omega_2$ is a linear combination of cycles only involving
arrows in $S_2$.  But then $\omega_1= \sum_{a \in S_1} a (\partial_{a} \omega)$ and $\omega_2 =  \sum_{a \in S_2} a (\partial_{a} \omega)$,
and so $\omega_1, \omega_2 \in VR$.  Now for each arrow $a$ we can also define a right sided operator, $\widetilde{\partial}_a$, which
sends a path $a_1 a_2 \dots a_n$ to $a_1 a_2 \dots a_{n-1}$ if $a_n= a$ and to $0$ otherwise.  It follows easily from the definition of twisted superpotential that the $\kk$-span of $\{ \widetilde{\partial}_a(\omega) | a \in Q_1 \}$ is the same space of relations $R$.
Thus a similar argument as above on the right shows that $\omega_1, \omega_2 \in RV$ as well.

Finally, since $Q$ is connected, there is some vertex $v$ incident to both an edge in $S_1$ and an edge in $S_2$.  Now note that every arrow
$a$ in $Q$ must occur in the superpotential $\omega$:  otherwise, the relation $\partial_a \omega$ is $0$, contradicting the independence
proved above.  In particular, since there is an arrow in $S_1$ whose head or tail is the vertex $v$, that arrow occurs in an $s$-cycle through $v$ consisting of arrows in $S_1$, which appears in $\omega$ with nonzero coefficient.  Similarly, there is an $s$-cycle through $v$ consisting of arrows in $S_2$,
which appears in $\omega$ with nonzero coefficient.  Finally, assume that $v$ is the vertex labeled $1$.  Then $e_1 \omega_1 e_1$ and $e_1 \omega_2 e_1$ are linearly independent elements in $e_1 (VR \cap RV) e_1$, contradicting the earlier part of the proof.

Thus $h_A(t)$ and $(I - Mt + Mt^{s-1} - I t^s)^{-1}$ disagree in some degree less than or equal to $s$, as claimed. Then $A$ cannot be twisted Calabi-Yau of dimension $3$ and type $(M, I, s)$ by Theorem~\ref{thm.hilb2}.
\end{proof}

\begin{corollary}
\label{cor.partition}
For each of the following matrices $M$, $(M, I, 3)$ cannot be the type of a twisted Calabi-Yau algebra of dimension $3$:
\[
\begin{psmatrix}2 & 2 & 0 \\ 1 & 1 & 1 \\ 1 & 0 & 1\end{psmatrix},
\begin{psmatrix}3 & 0 & 3 \\ 2 & 2 & 0 \\ 1 & 1 & 0\end{psmatrix},
\begin{psmatrix}2 & 0 & 3 \\ 10 & 2 & 0 \\ 1 & 4 & 0\end{psmatrix},
\begin{psmatrix}\alpha & x & 0 \\ 0 & 1 & 2 \\ x & 2 & 1\end{psmatrix},
\begin{psmatrix}3 & x & 0 \\ 0 & 0 & 2 \\ x & 0 & 0\end{psmatrix}.
\]
\end{corollary}
\begin{proof}
 In each case the quiver $Q$ with incidence matrix $M$ has a vertex $v$ with incident loops $a_1, \dots, a_m$
such that every $3$-cycle in $Q$ is either a composition of loops at $v$ or else does not involve any loops at $v$.  Thus, in the terminology of
Lemma~\ref{lem.partition}, one can take $S_1 = \{a_1, \dots, a_m \}$ and $S_2 = Q_1 - \{a_1, \dots, a_m \}$.
\end{proof}

The following example can be eliminated by a method similar to Lemma~\ref{lem.partition}, except examining relations of
degree equal to the superpotential that go from a vertex to a different vertex.
\begin{lemma}
There is no twisted Calabi-Yau algebra of dimension $3$ satisfying Hypothesis~\ref{hyp.scy} of type $(M = \begin{psmatrix}1 & 2 & 1 \\ 1 & 0 & 2 \\ 2 & 0 & 0\end{psmatrix}, I, 3)$.
\end{lemma}
\begin{proof}
Label the quiver as follows:
 
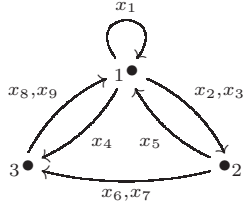
\[  \xymatrix{
& & {}_1\bullet \ar@/^/[dr]^{x_2,x_3} \ar@/^/[dl]^{x_4}
\ar@(ul,ur)^{x_1}[] & &  \\
& {}_3\bullet \ar@/^/[ur]^{x_8,x_9} &
& \bullet_2 \ar@/^/[ul]^{x_5} \ar@/^/[ll]^{x_6,x_7} & & }\]

The same argument as in Lemma~\ref{lem.partition} shows that if there is such a twisted Calabi-Yau algebra $A = \kk Q/(\partial_a \omega : a \in Q_1)$ for a twisted superpotential $\omega$, then if $V$ is the span of the arrows in $Q$ and $R$ is the span of the relations, we must have $e_3(RV \cap VR)e_2 = 0$ in degree $3$.

Write $\omega = c_1x_4x_8x_1 + c_2x_4x_9x_1 + c_3x_5x_1x_2 + c_4x_5x_1x_3 + \dots$, where only the four displayed terms will be relevant.
Since the displayed terms contain all 3-cycles beginning with $x_4$, $\delta_{x_4} \omega = c_1x_8x_1 + c_2x_9x_1$.  Similarly,
$\delta_{x_5} \omega = c_3x_1x_2 + c_4x_1x_3$.  Now $W_1 = (\delta_{x_4} \omega) (\kk x_2 + \kk x_3) \subseteq e_3 RV e_2$ and
$W_2 = (\kk x_8 + \kk x_9) (\delta_{x_5} \omega) \subseteq e_3 VR e_2$.  Note that either $W_1$ is $2$-dimensional or else
$W_1 = 0$, but the latter happens only if $\delta_{x_4} \omega = 0$, which as we saw in the proof of Lemma~\ref{lem.partition}, cannot happen.  So $\dim_{\kk} W_1 = 2$ and similarly, $\dim_{\kk} W_2 = 2$.  If $W_1 \cap W_2 = 0$ we will get
$W_1 + W_2 =  \kk x_8x_1x_2 +   \kk x_8x_1x_3 + \kk x_9x_1x_2 + \kk x_9x_1x_3$, but clearly
both $W_1$ and $W_2$ are contained in
\[
U = \{ ax_8x_1x_2 +  bx_8x_1x_3 + c x_9x_1x_2 + dx_9x_1x_3 : ad -bc = 0 \},
\]
so this is also a contradiction.  Thus $W_1 \cap W_2 \neq 0$, contradicting $e_3(RV \cap VR)e_2 = 0$ in degree $3$.
\end{proof}

The next quiver is eliminated by a more ad-hoc method, by examining the Gr{\"o}bner basis of a restriction to a subquiver
and seeing that we get a Hilbert series that is too large.
\begin{lemma}
\label{lem.sp1}
The following quiver $Q$ does not support an algebra satisfying Hypothesis \ref{hyp.scy} of type $(M, I, 3)$, where $M$ is the
incidence matrix of $Q$:
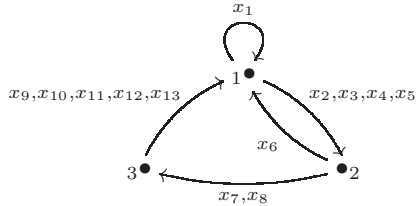
 \[  \xymatrix{
& & {}_1\bullet \ar@/^/[dr]^{x_2,x_3,x_4,x_5} \ar@(ul,ur)^{x_1}[] & &  \\
& {}_3\bullet \ar@/^/[ur]^{x_9,x_{10},x_{11},x_{12},x_{13}} &
& \bullet_2 \ar@/^/[ul]^{x_6} \ar@/^/[ll]^{x_7,x_8} & &}\]
\end{lemma}
\begin{proof}
If $A = \kk Q/(\partial_a \omega : a \in Q_1)$ is such a twisted Calabi-Yau algebra, where $\omega$ is a $\sigma$-twisted superpotential
such that $\sigma$ does not permute the vertices, then we know that $h_A(t) = I - Mt + M^T t^2 - I t^3$.
A calculation of this predictive Hilbert series shows that there should be 11 paths of length 4
from vertex 1 to itself modulo relations, in other words $\dim_{\kk} e_1 A_4 e_1 = 11$.
We will show by contrast that this space is at least $12$-dimensional.

Since the twist $\sigma$ does not permute the vertices, we must have $\sigma(x_1) = \rho^{-1} x_1$ and $\sigma(x_6) = \theta x_6$ for some nonzero constants $\rho, \theta$.  Thus $\omega = \omega' + \omega''$, where $\omega''$ is a sum of cycles which go through vertex $3$,
and
\begin{gather*}
\omega' =  c_1x_1^3 + c_2x_1x_2x_6 + c_3x_1x_3x_6 + c_4x_1x_4x_6 + c_5x_1x_5x_6 \\
+ \theta c_2 x_6 x_2 x_1 + \theta c_3 x_6 x_3 x_1 + \theta c_4 x_6 x_4 x_1 + \theta c_5 x_6 x_5 x_1 \\
+ \rho c_2 x_2 x_6 x_1 + \rho c_3 x_3 x_6 x_1 + \rho c_4 x_4 x_6 x_1 + \rho c_5 x_5 x_6 x_1.
\end{gather*}

Now consider the full subquiver $Q'$ on the vertices $1$ and $2$:
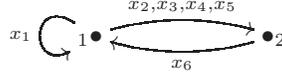
\[  \xymatrix{
{}_1\bullet \ar@/^/[rr]^{x_2,x_3,x_4,x_5} \ar@(ul,dl)_{x_1}[] & &  \bullet_2 \ar@/^/[ll]^{x_6} }\]
Let $A' = \kk Q'/I'$ be the restriction of $A$ to $\kk Q'$, as in Section~\ref{sec.key}.  It will suffice to show that $e_1 A'_4 e_2 \geq 12$,
since $A$ surjects onto $A'$.  Note that any cycle in $\omega''$ involves two arrows not in $Q'$,
and so any partial derivative of it still involves an arrow not in $Q'$, which becomes zero in the restriction.  Thus to calculate the relations
of $A'$ we can simply take partial derivatives of $\omega'$ with respect to arrows in $Q'$:
\begin{align*}
\partial_{x_1} \omega &= c_1 x_1^2 + c_2x_2x_6 + c_3x_3x_6 + c_4x_4x_6 + c_5x_5x_6, \\
\partial_{x_j} \omega &= \rho c_j x_6x_1 \;\;\; (j=2,\hdots,5), \\
\partial_{x_6} \omega &= \theta c_2x_1x_2 + \theta c_3x_1x_3 + \theta c_4x_1x_4 + \theta c_5x_1x_5.
\end{align*}
Note that all of the relations above given by $\partial x_j \omega$,
$j=2,\hdots,5$, are in fact the same because they differ only by a scalar, and similarly we can remove the $\theta$ from the last relation.

There are several cases depending on whether any of the $c_i$'s are $0$.  We do the following calculation assuming all $c_i$ are
nonzero.  The calculation of cases where some $c_i$'s may be $0$ is similar and we leave it to the reader.
Place an ordering on the generators $x_2 > x_3> x_4 > x_5> x_6 > x_1$
so that the leading terms are $x_2x_6, x_6x_1, x_1x_2$.
The overlap $x_6x_1x_2$ resolves. On the other hand, examining the overlap $x_2x_6 x_1$ we get
\begin{align*}
x_2(x_6x_1) & = 0, \\
(x_2x_6)x_1 & = -\frac{1}{c_2} \left( c_1x_1^2 + c_3x_3x_6 + c_4x_4x_6 + c_5x_5x_6\right)x_1
= -\frac{c_1}{c_2} x_1^3,
\end{align*}
forcing a new degree 3 relation $x_1^3=0$.
Finally, looking at the overlap $x_1x_2x_6$ gives
\begin{align*}
(x_1x_2)x_6  & = -\frac{1}{c_2}\left( c_3x_1x_3 + c_4x_1x_4 + c_5x_1x_5\right)x_6,  \\
x_1 (x_2x_6)  & = -\frac{1}{c_2} x_1\left( c_1x_1^2 + c_3x_3x_6 + c_4x_4x_6 + c_5x_5x_6\right),
\end{align*}
so using the relation $x_1^3 = 0$ this overlap resolves.  Similarly, the degree 4 overlaps $x_1^3 x_2$ and $x_6 x_1^3$ resolve,
so we have found a Gr{\"o}bner basis of relations for $A'$.
There are 29 paths in $Q'$ of length 4 from vertex $1$ to itself.
Removing those containing the leading terms $x_2x_6, x_6x_1, x_1x_2$ or $x_1^3$ results in 12 independent paths, which is
the desired contradiction.
\end{proof}

The following example is the most complicated one, and our analysis has been unable to conclusively eliminate the possibility
of a good twisted superpotential on it.  We can, however, prove that there is no good untwisted superpotential using a computational
method, as follows.  By \cite[Corollary 4.4]{Bo}, if there is any good superpotential on a quiver $Q$, a generic superpotential is good
(where generic means outside of a set of measure zero).  We employed the GAP \cite{GAP4} packages QPA \cite{QPA}
and GBNP \cite{GBNP} to compute the Hilbert series of a derivation-quotient algebra on $Q$ assuming a \textit{random} superpotential.
When a quiver $Q$ has no good superpotentials, this usually showed up in a consistent way, by the Hilbert series $h_A(t)$ being consistently
wrong for paths from $i$ to $j$ in the same degree, for a series of random choices.  For the quivers in several of the previous results, the results of this computer analysis then suggested those quivers were bad, and we later found the more formal reasons we presented above.   For the following quiver $Q$, however, the discrepancy in the Hilbert series doesn't show up until degree $6$.  We used the NCAlgebra package in Macaulay2, which was able to compute a Gr\"{o}bner basis for a generic superpotential, to show again that the Hilbert series is wrong.  Given that a twisted superpotential has quite a bit of extra freedom in its coefficients, this computationally intensive method does not easily apply to the twisted case.  Moreover, we do not have available a twisted version of Bocklandt's theorem which would allow us to claim that a generic twisted superpotential is good if any are (though such a theorem may well be true).
\begin{lemma}
\label{lem.outlier}
The following quiver $Q$ does not support an algebra $A$ satisfying Hypothesis \ref{hyp.scy} which is
(untwisted) Calabi-Yau of type $(M, I, 4)$, where $M$ is the incidence matrix of $Q$.
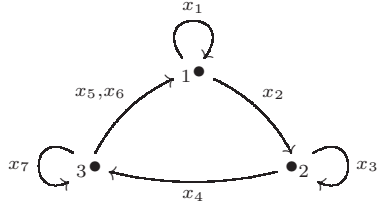
\[  \xymatrix{
& & {}_1\bullet \ar@/^/[dr]^{x_2} \ar@(ul,ur)^{x_1}[] & &  \\
& {}_3\bullet \ar@/^/[ur]^{x_5,x_6} \ar@(ul,dl)_{x_7}[] &
& \bullet_2  \ar@/^/[ll]^{x_4} \ar@(ur,dr)^{x_3}[] & &}
\]
\end{lemma}
\begin{proof}
The predictive Hilbert series $h_A(t)$ suggests there should be
13 paths of length 6 from vertex 3 to vertex 1.
A generic superpotential has the form
\begin{align*}
\omega &=
	c_1x_1^4 + c_2x_3^4 + c_3x_7^4
     + c_4x_1x_2x_4x_5 + c_5x_3x_4x_5x_2 + c_6x_7x_5x_2x_4 \\
     &\quad + c_7x_1x_2x_4x_6 + c_8x_3x_4x_6x_2 + c_9x_7x_6x_2x_4,
\end{align*}
for some $c_i \in \kk$, which we can assume are algebraically independent.
Taking cyclic derivatives gives seven degree 3 relations with one inclusion ambiguity to resolve.
There is one degree 4 overlap ambiguity that does not resolve
and six degree 5 overlap ambiguities that do not resolve.
This leaves two degree 6 overlap ambiguities from vertex 3 to vertex 1.
Both of these resolve and this leaves 14 paths of length 6 from vertex 3 to vertex 1.
So a generic superpotential on $Q$ is not good.  By \cite[Corollary 4.4]{Bo}, this means that no superpotential on $Q$
is good.
\end{proof}

Finally, we have the following isolated example which we can eliminate because it predicts the wrong growth.
\begin{lemma}
\label{lem.gkdim}
Suppose $Q$ is the quiver with adjacency matrix
\[
M_Q = \begin{psmatrix}2 & 3 & 0 \\ 1 & 0 & 2 \\ 4 & 0 & 0\end{psmatrix}.
\]
If $A=\kk Q/(\partial_a \omega : a \in Q_1, \deg(\omega)=3)$,
then the predictive Hilbert series indicates $\GKdim(A) \geq 5$.
Hence, $Q$ does not support an algebra satisfying Hypothesis \ref{hyp.scy}.
\end{lemma}
\begin{proof}
Let $p(t)$ denote the matrix polynomial corresponding to $M$. Then
\[\det(p(t)) = -(t^2+t+1) (t+1)^2 (t-1)^5\]
and
\[ s(t) = (t^2+3)(3t^2+1)(t+1)^2,\]
where $s(t)$ is the sum of the entries in $\adj(p(t))$.
Thus the (non-matrix) Hilbert series of $A$ is $s(t)/(\det p(t))$.
Thus, if $m_d$ (resp. $m_s$) denotes the multiplicity of vanishing of $\det(p(t))$
(resp. $s(t)$) at $t=1$, then by \cite[Lemma 2.7]{RR1} we have
$\GKdim(A) = m_d-m_s = 5$.
\end{proof}
\begin{remark}
Theoretically, one should be able to rule out the quiver in Lemma \ref{lem.gkdim} again by using Gr\"obner basis methods.
The Hilbert series computed using a random superpotential consistently
fails to match that of the predictive Hilbert series.
However, the quiver is too large for Macaulay2  to handle.
\end{remark}

\subsection{The final list}

Combining the results from this section, we have our final main theorem.
\begin{proposition}
\label{prop.3vert3}
Let $A \iso \kk Q/(\partial_a \omega : a \in Q_1)$ satisfy Hypothesis \ref{hyp.scy} with $|Q_0| = 3$,
where $A$ is of type $(M, I, s = \deg(\omega))$.   Then up to relabeling of the vertices in $Q$ or taking an opposite quiver,
$M$ is one of the matrices in the list below.  The indeterminates $a,b,c$ are positive integers satisfying the equation $a^2+b^2+c^2=abc$.
\begin{align*}
&\deg(\omega)=3: \quad
\begin{psmatrix}2 & 1 & 1 \\ 1 & 0 & 1 \\ 1 & 1 & 0\end{psmatrix},
\begin{psmatrix}2 & 1 & 1 \\ 1 & 1 & 0 \\ 1 & 0 & 1\end{psmatrix},
\begin{psmatrix}1 & 1 & 1 \\ 1 & 1 & 1 \\ 1 & 1 & 1\end{psmatrix},
\begin{psmatrix}1 & 1 & 1 \\ 1 & 2 & 0 \\ 1 & 0 & 2\end{psmatrix},
\begin{psmatrix}0 & a & 0 \\ 0 & 0 & b \\ c & 0 & 0\end{psmatrix}, \\
&\deg(\omega)=4: \quad
\begin{psmatrix}1 & 1 & 0 \\ 0 & 1 & 1 \\ 1 & 0 & 1\end{psmatrix},
\begin{psmatrix}1 & 1 & 1 \\ 1 & 0 & 0 \\ 1 & 0 & 0\end{psmatrix},
\begin{psmatrix}0 & 1 & 1 \\ 1 & 0 & 1 \\ 1 & 1 & 0\end{psmatrix},
\begin{psmatrix}1 & 0 & 1 \\ 0 & 1 & 1 \\ 1 & 1 & 0\end{psmatrix},
\begin{psmatrix}1 & 1 & 0 \\ 0 & 1 & 1 \\ 2 & 0 & 1\end{psmatrix}^*.
\end{align*}
Moreover, every matrix above, except possibly the final starred one, does occur in the type for some such algebra $A$.
The starred one does not occur if $\omega$ is an untwisted superpotential, so there is no Calabi-Yau algebra on that quiver.
\end{proposition}
\begin{proof}
The matrices listed are precisely the ones remaining from the long list given
in Lemma~\ref{lem.goodpoly} after removing the ones that cannot occur
for one of the reasons given in the preceding results.  We proved that the final starred example cannot occur if $\omega$
is untwisted, by Lemma~\ref{lem.outlier}.  So it remains only to prove that the other examples all occur.

A number of them occur as (algebras Morita equivalent to) skew group algebras $R \# \kk G$ where $R$ is Artin-Schelter regular
of dimension $3$ and $G = \mb{Z}_3$ or $G = S_3$.  See Tables~\ref{table.3ver3} and \ref{table.3ver4}.

We essentially constructed an example of type $(M', I, 4)$ where $M' = \begin{psmatrix}1 & 0 & 1 \\ 0 & 1 & 1 \\ 1 & 1 & 0\end{psmatrix}$
in Example~\ref{ex.component}.  Namely, let the group $G$ of Example~\ref{ex.component}
act on a cubic Artin-Schelter regular algebra $R$ of dimension $3$, and take an algebra $A$ Morita equivalent to a component of $R \# \kk G$.
For example, if one takes $R = \kk \langle x, y \rangle/(yx^2 + x^2y, y^2 x + xy^2)$, then the action of $G$ on $\kk x + \kk y$
given in Example~\ref{ex.component} respects the relations and thus gives an action of $G$ on $R$.  Since $R$ comes from
the superpotential $x^2 y^2 + x y^2 x + y^2 x^2 + y x^2 y$, it is easy to see that the action of $G$ has trivial homological determinant
and hence $\mu_A$ acts trivially on the vertices as desired.

As was also noted in the proof of Proposition~\ref{prop.3vert2}, the same group acting on $\kk[x,y]$ produces a Calabi-Yau algebra $B$ of dimension $2$ on the same quiver $Q'$ with incidence matrix $M$.  Then $A = B[t]$ is Calabi-Yau of dimension $3$ and type $(M = M' +I, I, 3)$ where $M = \begin{psmatrix}1 & 1 & 1 \\ 1 & 2 & 0 \\ 1 & 0 & 2\end{psmatrix}$, by the discussion of Ore extensions in Section~\ref{sec.motive}.

That leaves $M = \begin{psmatrix}0 & a & 0 \\ 0 & 0 & b \\ c & 0 & 0\end{psmatrix}$ where $a^2+b^2+c^2=abc$.
When $a = b = c = 3$, we obtain this type as a McKay quiver in Table~\ref{table.3ver3}.  By Proposition \ref{prop.mutation},
we get Calabi-Yau algebras of type $(M, I, 3)$ for all of the other $M$ by applying a series of quiver mutations.
\end{proof}

\begin{table}
\caption[3ver3]{$G = \mb{Z}_3 = \langle \sigma \rangle$, type $\left(M,\begin{psmatrix}1 &0 & 0 \\ 0 & 1 & 0 \\ 0 &  0 & 1 \end{psmatrix},\deg(\omega)\right)$
of $R \# \kk G$.   Let $\zeta$ be a primitive third root of unity, and let $\omega'(E)= y^3x + \zeta y^2xy + \zeta^2 yxy^2 + xy^3 + x^4$.}\label{table.3ver3}
\begin{tabular}{ccccc}\hline
$\sigma$ & $\omega$ & $M$ \\ \hline
$\operatorname{diag}(\zeta^2, \zeta^2, \zeta^2)$ & $(xyz +  zxy + yzx) +(- xzy - yxz - zyx)$ & $\left(\begin{smallmatrix}0  & 3 & 0 \\ 0 & 0 & 3 \\ 3 & 0 & 0 \end{smallmatrix}\right)$ \\
$\operatorname{diag}(1, \zeta, \zeta^2)$ & $(xyz +  zxy + yzx) +(- xzy - yxz - zyx)$ & $\left(\begin{smallmatrix}1  & 1 & 1 \\ 1 & 1 & 1 \\ 1 & 1 & 1 \end{smallmatrix}\right)$ \\
$\operatorname{diag}(\zeta, \zeta^2)$ & $y^2x^2 + xy^2x + x^2y^2 + yx^2y$ & $\left(\begin{smallmatrix}0   & 1 & 1 \\ 1 & 0 & 1 \\ 1 &1 & 0 \end{smallmatrix}\right)$ \\
$\operatorname{diag}(1, \zeta^2)$ & $\omega'(E)$  & $\left(\begin{smallmatrix}1   & 1 & 0 \\ 0 & 1 & 1 \\ 1 &0 & 1 \end{smallmatrix}\right)$
\end{tabular}
\end{table}

\begin{table}
\caption[3ver4]{$G = S_3 = \langle \sigma, \tau \rangle$, $\dim V = 3$, where $\sigma = (123)$ and $\tau = (12)$,
type $\left(M,\begin{psmatrix}1 & 0 & 0 \\ 0 & 1 & 0 \\ 0 &  0 & 1 \end{psmatrix},\deg(\omega)\right)$ of $R \# \kk G$}\label{table.3ver4}
\begin{tabular}{cccc}\hline
$\sigma$ & $\tau$ & $\omega$ & $M$ \\ \hline
$\left(\begin{smallmatrix}0  & 0 & 1 \\ 1 & 0 & 0 \\ 0 & 1 & 0 \end{smallmatrix}\right)$ &
$\left(\begin{smallmatrix}0 & 1 &  0 \\1 & 0 & 0 \\ 0 & 0 & 1 \end{smallmatrix}\right)$ & $(xyz +  zxy + yzx) +(xzy+ yxz + zyx)$ & $\left(\begin{smallmatrix}2 & 1  &  1 \\1 & 1 & 0 \\ 1 & 0 & 1 \end{smallmatrix}\right)$ \\
$\left(\begin{smallmatrix}0  & 0 & 1 \\ 1 & 0 & 0 \\ 0 & 1 & 0 \end{smallmatrix}\right)$ &
$\left(\begin{smallmatrix}0 & -1 &  0 \\-1 & 0 & 0 \\ 0 & 0 & -1 \end{smallmatrix}\right)$ & $(xyz +  zxy + yzx) +(- xzy - yxz - zyx)$ & $\left(\begin{smallmatrix} 2 &1  &  1 \\1 & 0 & 1 \\ 1 & 1 & 0 \end{smallmatrix}\right)$ \\
 $\operatorname{diag}(\zeta, \zeta^2)$ &
$\left(\begin{smallmatrix}0 & 1 \\1 & 0 \end{smallmatrix}\right)$ & $y^2x^2 + xy^2x + x^2y^2 + yx^2y$ &
$\left(\begin{smallmatrix}1 &1  &  1 \\ 1 & 0 & 0 \\ 1 & 0 & 0 \end{smallmatrix}\right)$
\end{tabular}
\end{table}

\appendix
\section{McKay Quivers}
\label{sec:app}

We list all McKay quivers $Q$ with at most three vertices for groups $G$ acting on
the polynomial ring with two or three variables, where the degree one elements give a representation $V$
of $G$.  Recall that by our convention in this paper, if we label the vertices of the quiver by the
the irreducible characters $\chi_i$ of $G$, then if $V$ has character $\rho$, and we write
$\rho \chi_j= \sum_i a_{ij}  \chi_i$, then $M = M_Q = (a_{ij})$ is the incidence matrix of $Q$.

Since we want at most three vertices, the McKay quiver must be derived from a group with
at most three irreducible representations.
Hence, in this case the possible groups are $\ZZ_2$, $\ZZ_3$ and $S_3$.
Moreover, the defining representation $V$ must be either two- or three-dimensional, respectively.
For reference, we include the character tables of our groups below.
We represent the character $\rho$ of $V$ as a tuple $(a_1, \dots, a_m)$ where $\rho = \chi_{a_1} + \dots + \chi_{a_m}$.  Below, $\zeta$ represents a third root of unity.

\begin{multicols}{3}
$\ZZ_2 = \langle g \rangle$

\begin{tabular}{c|cc}
& $e$ & $g$ \\ \hline
$\chi_1$ & 1 & 1 \\
$\chi_2$ & 1 & -1
\end{tabular}

\columnbreak

$G=\ZZ_2$, $\dim V = 2$.

\begin{tabular}{cc}\hline
$\rho$ & $M_Q$ \\ \hline
( 1, 1 ) & $\left(\begin{smallmatrix}2 & 0 \\ 0 & 2 \end{smallmatrix}\right)$ \\
( 1, 2 ) & $\left(\begin{smallmatrix}1 & 1 \\ 1 & 1 \end{smallmatrix}\right)$\\
( 2, 2 ) & $\left(\begin{smallmatrix}0 & 2 \\ 2 & 0 \end{smallmatrix}\right)$
\end{tabular}

\columnbreak

$G=\ZZ_2$, $\dim V = 3$.

\begin{tabular}{cc}\hline
$\rho$ & $M_Q$ \\ \hline
( 1, 1, 1 ) & $\left(\begin{smallmatrix}3 & 0 \\ 0 & 3 \end{smallmatrix}\right)$\\
( 1, 1, 2 ) & $\left(\begin{smallmatrix}2 & 1 \\ 1 & 2 \end{smallmatrix}\right)$\\
( 1, 2, 2 ) & $\left(\begin{smallmatrix}1 & 2 \\ 2 & 1 \end{smallmatrix}\right)$\\
( 2, 2, 2 ) & $\left(\begin{smallmatrix}0 & 3 \\ 3 & 0 \end{smallmatrix}\right)$
\end{tabular}
\end{multicols}

\begin{multicols}{3}
$\ZZ_3=\langle g \rangle$

\begin{tabular}{c|ccc}
			& $e$ & $g$ & $g^2$ \\ \hline
$\chi_1$		& 1	& 1	& 1 \\
$\chi_2$		& 1	& $\zeta$	& $\zeta^2$ \\
$\chi_3$		& 1	& $\zeta^2$	 & $\zeta$
\end{tabular}

\columnbreak

$G=\ZZ_3$ $\dim V = 2$.

\begin{tabular}{cc}\hline
$\rho$ & $M_Q$ \\ \hline
( 1, 1 ) & $\left(\begin{smallmatrix}2 & 0 & 0 \\ 0 & 2 & 0  \\ 0 & 0 & 2\end{smallmatrix}\right)$ \\
( 1, 2 ) & $\left(\begin{smallmatrix}1 & 0 & 1 \\ 1 & 1 & 0 \\ 0 & 1 & 1\end{smallmatrix}\right)$ \\
( 1, 3 ) & $\left(\begin{smallmatrix}1 & 1 & 0 \\ 0 & 1 & 1 \\ 1 & 0 & 1\end{smallmatrix}\right)$\\
( 2, 2 ) & $\left(\begin{smallmatrix}0 & 0 & 2 \\ 2 & 0 & 0 \\ 0 & 2 & 0\end{smallmatrix}\right)$\\
( 2, 3 ) & $\left(\begin{smallmatrix}0 & 1 & 1 \\ 1 & 0 & 1 \\ 1 & 1 & 0\end{smallmatrix}\right)$\\
( 3, 3 ) & $\left(\begin{smallmatrix}0 & 2 & 0 \\ 0 & 0 & 2 \\ 2 & 0 & 0\end{smallmatrix}\right)$
\end{tabular}

\columnbreak

$G=\ZZ_3$, $\dim V = 3$.

\begin{tabular}{cc} \hline
$\rho$ & $M_Q$ \\ \hline
( 1, 1, 1 ) & $\left(\begin{smallmatrix}3 & 0 & 0 \\ 0 & 3 & 0 \\ 0 & 0 & 3\end{smallmatrix}\right)$\\
( 1, 1, 2 ) & $\left(\begin{smallmatrix}2 & 0 & 1 \\ 1 & 2 & 0 \\ 0 & 1 & 2\end{smallmatrix}\right)$\\
( 1, 1, 3 ) & $\left(\begin{smallmatrix}2 & 1 & 0 \\ 0 & 2 & 1 \\ 1 & 0 & 2\end{smallmatrix}\right)$\\
( 1, 2, 2 ) & $\left(\begin{smallmatrix}1 & 0 & 2 \\ 2 & 1 & 0 \\ 0 & 2 & 1\end{smallmatrix}\right)$\\
( 1, 2, 3 ) & $\left(\begin{smallmatrix}1 & 1 & 1 \\ 1 & 1 & 1 \\ 1 & 1 & 1\end{smallmatrix}\right)$\\
( 1, 3, 3 ) & $\left(\begin{smallmatrix}1 & 2 & 0 \\ 0 & 1 & 2 \\ 2 & 0 & 1\end{smallmatrix}\right)$\\
( 2, 2, 2 ) & $\left(\begin{smallmatrix}0 & 0 & 3 \\ 3 & 0 & 0 \\ 0 & 3 & 0\end{smallmatrix}\right)$\\
( 2, 2, 3 ) & $\left(\begin{smallmatrix}0 & 1 & 2 \\ 2 & 0 & 1 \\ 1 & 2 & 0\end{smallmatrix}\right)$\\
( 2, 3, 3 ) & $\left(\begin{smallmatrix}0 & 2 & 1 \\ 1 & 0 & 2 \\ 2 & 1 & 0\end{smallmatrix}\right)$\\
( 3, 3, 3 ) & $\left(\begin{smallmatrix}0 & 3 & 0 \\ 0 & 0 & 3 \\ 3 & 0 & 0\end{smallmatrix}\right)$
\end{tabular}
\end{multicols}

\begin{multicols}{3}
$S_3=\langle (12),(123) \rangle$

\begin{tabular}{c|ccc}
			& $e$ & $(12)$ & $(123)$ \\ \hline
$\chi_1$	& 2	& $0$ & $-1$ 	 \\
$\chi_2$	& 1	& $-1$	& $1$	\\
$\chi_3$	& 1	& 1	& 1
\end{tabular}

\columnbreak

$G=S_3$, $\dim V = 2$.

\begin{tabular}{cc}\hline
$\rho$ & $M_Q$ \\ \hline
(1)    & $\left(\begin{smallmatrix}1 & 1 & 1 \\ 1 & 0 &0 \\ 1 & 0 & 0\end{smallmatrix}\right)$\\
( 3, 3 ) & $\left(\begin{smallmatrix}2 & 0 & 0 \\ 0 & 2 & 0 \\ 0 & 0 & 2\end{smallmatrix}\right)$\\
( 2, 3 ) & $\left(\begin{smallmatrix}2 & 0 & 0 \\ 0 & 1 & 1 \\ 0 & 1&1\end{smallmatrix}\right)$\\
( 2, 2 ) & $\left(\begin{smallmatrix}2 & 0 & 0 \\ 0 & 0 & 2 \\ 0 & 2 &0 \end{smallmatrix}\right)$
\end{tabular}

\columnbreak

$G=S_3$, $\dim V = 3$.

\begin{tabular}{cc}\hline
$\rho$ & $M_Q$ \\ \hline
( 1, 3 ) 	& $\left(\begin{smallmatrix}2 & 1 & 1 \\ 1 & 1 & 0 \\ 1 & 0 & 1\end{smallmatrix}\right)$\\
( 1, 2 ) 	& $\left(\begin{smallmatrix}2 & 1 & 1 \\ 1 & 0 & 1 \\ 1 & 1 & 0\end{smallmatrix}\right)$\\
(3, 3, 3 ) & $\left(\begin{smallmatrix}3 & 0 & 0 \\ 0 & 3 & 0 \\ 0 & 0 & 3\end{smallmatrix}\right)$\\
( 2, 3, 3 ) & $\left(\begin{smallmatrix}3 & 0 & 0 \\ 0 & 2 & 1 \\ 0 & 1 & 2\end{smallmatrix}\right)$\\
( 2, 2, 3 ) & $\left(\begin{smallmatrix}3 & 0 & 0 \\ 0 & 1 & 2 \\ 0 & 2 & 1\end{smallmatrix}\right)$\\
( 2, 2, 2 ) & $\left(\begin{smallmatrix}3 & 0 & 0 \\ 0 & 0 & 3 \\ 0 & 3 & 0\end{smallmatrix}\right)$
\end{tabular}

\end{multicols}



\begin{thebibliography}{10}

\bibitem{ATV1}
M.~Artin, J.~Tate, and M.~Van~den Bergh.
\newblock Some algebras associated to automorphisms of elliptic curves.
\newblock In {\em The {G}rothendieck {F}estschrift, {V}ol.\ {I}}, volume~86 of
  {\em Progr. Math.}, pages 33--85. Birkh\"auser Boston, Boston, MA, 1990.

\bibitem{AS}
Michael Artin and William~F. Schelter.
\newblock Graded algebras of global dimension {$3$}.
\newblock {\em Adv. in Math.}, 66(2):171--216, 1987.

\bibitem{ASS}
Ibrahim Assem, Daniel Simson, and Andrzej Skowro\'nski.
\newblock {\em Elements of the representation theory of associative algebras.
  {V}ol. 1}, volume~65 of {\em London Mathematical Society Student Texts}.
\newblock Cambridge University Press, Cambridge, 2006.

\bibitem{BBKNZ}
Georgia Benkart, Rekha Biswal, Ellen Kirkman, Van~C Nguyen, and Jieru Zhu.
\newblock Mckay matrices for finite-dimensional {H}opf algebras.
\newblock {\em arXiv preprint:2007.05510}, 2020.

\bibitem{Bo}
Raf Bocklandt.
\newblock Graded {C}alabi {Y}au algebras of dimension 3.
\newblock {\em J. Pure Appl. Algebra}, 212(1):14--32, 2008.

\bibitem{BSW}
Raf Bocklandt, Travis Schedler, and Michael Wemyss.
\newblock Superpotentials and higher order derivations.
\newblock {\em J. Pure Appl. Algebra}, 214(9):1501--1522, 2010.

\bibitem{butin}
F.~Butin.
\newblock Branching law for the finite subgroups of {${\bf SL}_4\mathbb{C}$}
  and the related generalized {P}oincar\'e polynomials.
\newblock {\em Ukrainian Math. J.}, 67(10):1484--1497, 2016.
\newblock Reprint of Ukra\"\i n. Mat. Zh. {{\bf{6}}7} (2015), no. 10,
  1321--1332.

\bibitem{CHI}
Daniel Chan, Paul Hacking, and Colin Ingalls.
\newblock Canonical singularities of orders over surfaces.
\newblock {\em Proc. Lond. Math. Soc. (3)}, 98(1):83--115, 2009.

\bibitem{CKWZ}
Kenneth Chan, Ellen Kirkman, Chelsea Walton, and James~J. Zhang.
\newblock Quantum binary polyhedral groups and their actions on quantum planes.
\newblock {\em J. Reine Angew. Math.}, 719:211--252, 2016.

\bibitem{GBNP}
A.M. Cohen and J.W. Knopper.
\newblock {GBNP} -- computing {G}r\"obner bases of noncommutative polynomials,
  version 1.0.1.
\newblock \url{http://www.gap-system.org/Packages/gbnp.html}, 2010.

\bibitem{dwz}
Harm Derksen, Jerzy Weyman, and Andrei Zelevinsky.
\newblock Quivers with potentials and their representations. {I}. {M}utations.
\newblock {\em Selecta Math. (N.S.)}, 14(1):59--119, 2008.

\bibitem{EE}
Pavel Etingof and Ching-Hwa Eu.
\newblock Koszulity and the {H}ilbert series of preprojective algebras.
\newblock {\em Math. Res. Lett.}, 14(4):589--596, 2007.

\bibitem{FZ1}
Sergey Fomin and Andrei Zelevinsky.
\newblock Cluster algebras. {I}. {F}oundations.
\newblock {\em J. Amer. Math. Soc.}, 15(2):497--529, 2002.

\bibitem{FZ2}
Sergey Fomin and Andrei Zelevinsky.
\newblock Cluster algebras. {II}. {F}inite type classification.
\newblock {\em Invent. Math.}, 154(1):63--121, 2003.

\bibitem{G}
Shlomo Gelaki.
\newblock Semisimple triangular {H}opf algebras and {T}annakian categories.
\newblock In {\em Arithmetic fundamental groups and noncommutative algebra
  ({B}erkeley, {CA}, 1999)}, volume~70 of {\em Proc. Sympos. Pure Math.}, pages
  497--515. Amer. Math. Soc., Providence, RI, 2002.

\bibitem{Gi}
Victor Ginzburg.
\newblock Calabi-yau algebras.
\newblock {\em arXiv preprint math/0612139}, 2006.

\bibitem{GS}
E.~S. Golod and I.~R. \v{S}afarevi\v{c}.
\newblock On the class field tower.
\newblock {\em Izv. Akad. Nauk SSSR Ser. Mat.}, 28:261--272, 1964.

\bibitem{GAP4}
The~GAP Group.
\newblock {GAP} -- {G}roups, algorithms, and programming, version 4.7.7.
\newblock \url{http://www.gap-system.org}, 2015.

\bibitem{HJ}
Roger~A. Horn and Charles~R. Johnson.
\newblock {\em Matrix analysis}.
\newblock Cambridge University Press, Cambridge, second edition, 2013.

\bibitem{IR}
Osamu Iyama and Idun Reiten.
\newblock Fomin-{Z}elevinsky mutation and tilting modules over {C}alabi-{Y}au
  algebras.
\newblock {\em Amer. J. Math.}, 130(4):1087--1149, 2008.

\bibitem{JZ}
Peter J{\o}rgensen and James~J. Zhang.
\newblock Gourmet's guide to {G}orensteinness.
\newblock {\em Adv. Math.}, 151(2):313--345, 2000.

\bibitem{KKZ}
E.~Kirkman, J.~Kuzmanovich, and J.~J. Zhang.
\newblock Gorenstein subrings of invariants under {H}opf algebra actions.
\newblock {\em J. Algebra}, 322(10):3640--3669, 2009.

\bibitem{LWW}
Liyu Liu, Shengqiang Wang, and Quanshui Wu.
\newblock Twisted {C}alabi-{Y}au property of {O}re extensions.
\newblock {\em J. Noncommut. Geom.}, 8(2):587--609, 2014.

\bibitem{M}
Susan Montgomery.
\newblock {\em Hopf algebras and their actions on rings}, volume~82 of {\em
  CBMS Regional Conference Series in Mathematics}.
\newblock Published for the Conference Board of the Mathematical Sciences,
  Washington, DC; by the American Mathematical Society, Providence, RI, 1993.

\bibitem{MS}
Izuru Mori and S.~Paul Smith.
\newblock {$m$}-{K}oszul {A}rtin-{S}chelter regular algebras.
\newblock {\em J. Algebra}, 446:373--399, 2016.

\bibitem{PP}
Alexander Polishchuk and Leonid Positselski.
\newblock {\em Quadratic algebras}, volume~37 of {\em University Lecture
  Series}.
\newblock American Mathematical Society, Providence, RI, 2005.

\bibitem{QPA}
The QPA-team.
\newblock {QPA} -- {Q}uivers, path algebras and representations, version 1.21.
\newblock \url{http://www.math.ntnu.no/~oyvinso/QPA/}, 2015.

\bibitem{RR2}
M.~Reyes and D.~Rogalski.
\newblock Graded twisted {C}alabi-{Y}au algebras are generalized
  {A}rtin-{S}chelter regular.
\newblock {\em arXiv preprint:1807.10249}, 2018, to appear in the Nagoya Mathematical Journal.

\bibitem{RRZ}
Manuel Reyes, Daniel Rogalski, and James~J. Zhang.
\newblock Skew {C}alabi-{Y}au algebras and homological identities.
\newblock {\em Adv. Math.}, 264:308--354, 2014.

\bibitem{RR1}
Manuel~L. Reyes and Daniel Rogalski.
\newblock Growth of graded twisted {C}alabi-{Y}au algebras.
\newblock {\em J. Algebra}, 539:201--259, 2019.

\end{thebibliography}
\bibliographystyle{plain}

\end{document}